\documentclass[11pt,reqno]{amsproc}
\usepackage[margin=1in]{geometry}
\usepackage{amsmath, amsthm, amssymb}
\usepackage[colorlinks=true, pdfstartview=FitV, linkcolor=blue,citecolor=blue, urlcolor=blue]{hyperref}
\usepackage[abbrev,lite,nobysame]{amsrefs}
\usepackage{times}
\usepackage[usenames,dvipsnames]{color}

\usepackage{mathtools}

\mathtoolsset{showonlyrefs=true}
\usepackage{dsfont}

\def\eps{\varepsilon}
\def\e{{\rm e}}

\def\Re{{\rm Re}}
\def\dd{{\rm d}}
\def\uu {\boldsymbol{u}}
\def\vv {\boldsymbol{v}}
\def\ff {{\widehat f}}
\def\ddt{{\frac{\dd}{\dd t}}}

\def\R {\mathbb{R}}
\def\u {\boldsymbol{u}}

\def\N {\mathbb{N}}
\def\H {{\dot H}}

\def\ff {{\widehat f}}

\def\ZZ {{\mathbb Z}}

\def \l {\langle}
\def \r {\rangle}
\def\T {{\mathbb T}}

\def\de{{\partial}}

\newtheorem{proposition}{Proposition}[section]
\newtheorem{theorem}[proposition]{Theorem}
\newtheorem{corollary}[proposition]{Corollary}
\newtheorem{lemma}[proposition]{Lemma}
\theoremstyle{definition}

\newtheorem{remark}[proposition]{Remark}

\numberwithin{equation}{section}

\title[Enhanced dissipation time-scales and mixing rates]{On the relation between enhanced dissipation time-scales and mixing rates}

\author[M. Coti Zelati, M. G. Delgadino and T. M. Elgindi]{Michele Coti Zelati, Matias G. Delgadino and Tarek M. Elgindi}

\address{Department of Mathematics, Imperial College London, London, SW7 2AZ, UK}
\email{m.coti-zelati@imperial.ac.uk}
\email{m.delgadino@imperial.ac.uk}

\address{Department of Mathematics, UC San Diego, La Jolla, CA 92093}
\email{telgindi@ucsd.edu}


\begin{document}

\begin{abstract}
We study diffusion and mixing in different linear fluid dynamics models, mainly related to incompressible flows. 
In this setting, mixing is a purely advective effect which causes a transfer of energy to high frequencies. 
When diffusion is present, mixing enhances the dissipative forces. This phenomenon is referred to as enhanced dissipation, namely
the identification of a time-scale faster than the purely diffusive one. 
We establish a precise connection between quantitative mixing rates in terms of decay of negative Sobolev norms
and enhanced dissipation time-scales. The proofs are based on a contradiction argument that takes advantage of
the cascading mechanism due to mixing, an estimate of the distance between the inviscid and viscous dynamics, and
of an optimization step in the frequency cut-off. 

Thanks to the generality and robustness of our approach, we are able to apply our abstract results to a number of problems. 
For instance, we prove that contact Anosov flows obey logarithmically fast dissipation time-scales. To the best of our knowledge, 
this is the first example of a flow that induces an enhanced dissipation time-scale faster than polynomial.
Other applications include passive scalar evolution in both planar and radial settings and fractional diffusion.

\end{abstract}


\maketitle


\section{Introduction}
This article deals with the so-called mixing/enhanced dissipation mechanism
in a large class of fluids and hydrodynamic stability problems. In general, the two main sources of mixing of a substance in
a liquid, or of a liquid with itself, are diffusion and advection. In many instances, advection is responsible for a faster dissipation rate, giving 
rise to the so-called enhanced dissipation/diffusion effect. From the mathematical viewpoint, the precise
quantification of this phenomenon for general flows is a very challenging problem. 

To fix ideas, let us consider a two-dimensional periodic domain $\T^2$ and a passive scalar $f:[0,\infty)\times \T^2\to \R$ 
that is advected by a smooth divergence-free (i.e. incompressible) velocity field $\uu:\T^2\to \R^2$, and therefore satisfies the initial-value problem
\begin{align}\label{eq:inviscidpass}
\begin{cases}
\de_tf+\uu\cdot\nabla f=0,\\
 f(0)=f^{in},
\end{cases}
\end{align}
for a \emph{mean-free} initial datum $f^{in}\in L^2$. The goal of this paper is to clarify and \emph{quantify} the connection between the 
decay properties of the solution to \eqref{eq:inviscidpass}, and those of its viscous counterpart
\begin{align}\label{eq:viscouspass}
\begin{cases}
\de_t f^\nu+ \uu\cdot\nabla f^\nu=\nu \Delta f^\nu,\\
f^\nu(0)=f^{in},
\end{cases}
\end{align}
where $\nu\in(0,1)$ is the diffusivity coefficient, proportional to the inverse P\'eclet number. Hereafter, we make more precise the concept
of \emph{mixing} for \eqref{eq:inviscidpass} and we explain how it acts to enhance the dissipative
forces. This phenomenon is referred to as \emph{enhanced dissipation}, namely
the identification of a time-scale faster than the purely diffusive one.
We remark here that \eqref{eq:inviscidpass}-\eqref{eq:viscouspass} constitute
only a special case of our general setting (see Section \ref{sec:abstract}), which covers
a variety of examples (see Section \ref{sub:examples}).

\subsection{Mixing in fluid flows}
In general, mixing refers to a cascading mechanism that transfers information (such as energy, enstrophy, etc.) 
to smaller and smaller spatial scales (or higher and higher frequencies), in a way that is time reversible and conservative for 
finite times but results in an irreversible loss of information as $t \to \infty$. The precise
definition of mixing really depends on the physics of the problem under study, and its quantification for infinite-dimensional systems 
is fundamental to deeply understand the dynamics. It is fairly accurate to think of mixing as a stabilizing mechanism for certain
stationary structures. It generates damping effects:  in kinetic theory this is known as \emph{Landau damping}, and was only recently
understood in a mathematically rigorous way in \cite{MV11}, 
while in fluid dynamics it is called \emph{inviscid damping} (studied even more recently in \cites{BM15,BMV14} at the nonlinear
level).

A way to quantify mixing for \eqref{eq:inviscidpass} is via the homogeneous Sobolev $\H^{-1}$ norm, as it provides an
averaged measure of the characteristic length-scale of the oscillations of the solution \cites{LTD11}. The velocity field $\uu$ is
called mixing if for every $f^{in}\in \H^1$, the corresponding solution $f$ to \eqref{eq:inviscidpass} satisfies
\begin{align}\label{eq:decayH1}
\lim_{t\to\infty}\|f(t)\|_{\H^{-1}}=0.
\end{align}
This definition agrees with the definition of mixing in the sense of ergodic theory (see \cite{LTD11}), and is equivalent to 
$f(t)$ being weakly convergent in $L^2$ to 0 as $t\to\infty$, due to the conservation of the $L^2$ norm. In turn, this describes the idea of a transfer of information to
higher frequencies, since it means that for every $n\in \ZZ^2\setminus\{0\}$ each Fourier mode $\ff_n(t)$ vanishes
in the long-time limit.

There have been numerous recent contributions to this field: without any aim of completeness, we mention the works on 
lower bounds on mixing rates \cites{IKX14,Seis13,LLNMD12}, mixing and regularity \cites{Jabin16,YZ17,Bressan03}, and,
more relevant for our discussion,
quantification of mixing rates in passive scalars  \cites{ACM16,CLS17,Zillinger18, BCZ15}  (see also \cites{Liv04,BDL18,sinai1961geodesic,anosov1967some} and references therein for a dynamical system viewpoint)
and two-dimensional Euler equations linearized around shear flows \cites{Zillinger14,Zillinger15,WZZ15,WZZkolmo17,WZZ17,GNRS18} and vortices \cites{CZZ18,Zcirc17,BCZV17}. 

In this article, we are interested in understanding how a quantification of \eqref{eq:decayH1} through 
a precise decay rate (either polynomial or exponential) affects the dissipative properties of the viscous problem \eqref{eq:viscouspass}.
We describe this in the next paragraph.

\subsection{Enhanced dissipation in fluid flows}
As it turns out, mixing is intimately connected with the so-called \emph{enhanced dissipation} effects due to the presence of
a fluid flow, which speeds up the dissipation rate and induces a dissipation time-scale faster than that
of the diffusion one alone. Let us consider the solution $f^\nu$ to the advection-diffusion equation \eqref{eq:viscouspass}. 
A simple $L^2$ estimate that uses the incompressibility of $\uu$ and the Poincar\'e 
inequality reads
\begin{align}
\|f^\nu(t)\|_{L^2}\leq \|f^{in}\|_{L^2}\e^{-\nu t}.
\end{align}
In fact, not much information is used about the flow $\uu$ other than incompressibility, and the decay rate is that of the heat equation, 
from which the natural \emph{diffusive}  time-scale $O(\nu^{-1})$ appears.
However, for small $\nu$, we expect the inviscid mixing to be the leading order dynamics (at least 
for some time) and hence we can predict a faster decay rate than the one prescribed by the heat equation. 
The important behavior to detect consists of a cascading mechanism due to the inviscid mixing, 
and its interaction with a small diffusion of order $\nu$. 
This effect has been called alternatively  \emph{shear-diffuse mechanism}, 
and has been studied many times in linear and some nonlinear settings in both mathematics \cites{Deng2013,BGM15III,CKRZ08,BCZGH15,BCZ15,BMV14,BVW16,BGM15I,LiWeiZhang2017,BW13,Gallay2017,IMM17,WZZkolmo17,BGM15II,WZ18,LWZ18} and physics \cites{BajerEtAl01,DubrulleNazarenko94,RhinesYoung83,LatiniBernoff01}.

A rigorous mathematical framework for so-called relaxation enhancing
flows $\uu$ has been developed in \cite{CKRZ08}.
Roughly speaking, a velocity field $\uu$ is relaxation enhancing if by the diffusive time-scale  $O(\nu^{-1})$,
arbitrarily much energy is already dissipated. The main result of \cite{CKRZ08} characterizes relaxation
enhancing flows  in terms of the spectral properties of the operator $\uu\cdot \nabla$.
Precisely, $\uu$ is relaxation enhancing if and only if the operator $\uu\cdot \nabla$ has no nontrivial eigenfunctions in $\H^1$. 
In particular, weakly mixing flows (i.e., those with only continuous spectrum) fall in this class. 
The proof of this result is based on the so-called RAGE theorem \cite{RS79-3} and, at this level of generality, contains no quantitative information
on the appearance of a faster time-scale induced by $\uu$.  This work is a quantitative revisitation of the arguments in \cite{CKRZ08}, avoiding 
the use of the RAGE theorem by requiring a decay rate for \eqref{eq:decayH1}.

A sufficient condition for $\uu$ to be relaxation enhancing  is that the corresponding passive scalar $f^\nu$ 
obeys an estimate of the type
\begin{align}
\|f^\nu(t)\|_{L^2}\leq \varrho(\nu^{q} t)\|f^{in}\|_{L^2}, \qquad \forall t > \frac{1}{\nu^{q}},
\end{align}
for some $q\in (0,1)$ depending on $\uu$ 
and some monotonically decreasing function $\varrho:[0,\infty)\to [0,\infty)$ vanishing at infinity. In this
case, it is apparent the $\uu$ induces a faster time-scale  $O(\nu^{-q})$. In general, the dependence
of $q$ on $\uu$ is very hard to detect.
For example, Kelvin showed in \cite{Kelvin87} that $x$-dependent modes of the linearized Couette  flow $\uu=(y,0)$ on $\T \times \R$ 
decay on a time-scale  $O(\nu^{-1/3})$. In \cite{BW13}, a similar result was proven for a passive scalar advected by 
the Kolmogorov flow $\uu=(\sin y, 0)$, with a time-scale   $O(\nu^{-1/2})$.  The case of general shear $\uu=(u(y),0)$ with 
a finite number of critical points was treated in \cite{BCZ15}: the enhanced dissipation time-scale was proved to be  
$O(\nu^{-(n_0+1)/(n_0+3)})$, where  $n_0\in \N$ 
denotes the maximal order of vanishing of $u'$ at the critical points. Recently, the case of the two-dimensional Navier-Stokes
equations linearized around the Kolmogorov flow was analyzed in \cites{LX17,WZZkolmo17,IMM17}, obtaining a time-scale
of $O(\nu^{-1/2})$, while \cite{GNRS18} deals with certain monotone shears on $\T \times \R$, which induce  a time-scale
$O(\nu^{-1/3})$. In the radial setting, we refer to \cites{Deng2013,LiWeiZhang2017,Gallay2017} for the most recent results on 
the 2D Navier-Stokes equations linearized around the Oseen vortex.

In all the results mentioned above, the connection between mixing (in the sense of \eqref{eq:decayH1}) and enhanced dissipation
is not made explicit. In particular, the knowledge of mixing decay rates is never used in the proofs. However, the energy balance
for \eqref{eq:viscouspass}, namely
\begin{align}
\ddt \|f^{\nu}\|^2_{L^2}+2\nu \| f^\nu\|^2_{\H^1}=0,
\end{align}
implies that an anomalous rate of growth $\| f^\nu(t)\|_{\H^1}$ corresponds to a faster decay for $\|f^{\nu}(t)\|_{L^2}$. Heuristically speaking,
at the inviscid level, the decay of the $\H^{-1}$ norm \eqref{eq:decayH1} implies growth of $\| f(t)\|_{\H^1}$ due to $L^2$ conservation. If the viscous and inviscid dynamics
are close, then the growth of $\| f^\nu(t)\|_{\H^1}$ follows. The main objective of this paper is to make this heuristics precise and quantitative,
by investigating the following problem.

\subsection{Fundamental problem}\label{sub:funprob}
Assume that solutions to \eqref{eq:inviscidpass} satisfy the mixing estimate
\begin{align}\label{eq:mixing}
\| f(t)\|_{\H^{-1}} \leq \varrho(t)\| f^{in}\|_{\H^1}, \qquad \forall t\geq0,\ \forall f^{in}\in \H^1,
\end{align}
for some monotonically decreasing function $\varrho:[0,\infty)\to [0,\infty)$ vanishing at infinity. Prove
that $\uu$ is relaxation enhancing and identify an enhanced-dissipation time-scale faster than the diffusive one
(e.g., $O(\nu^{-q})$ for some $q\in(0,1)$, or $O(|\ln\nu|^{q})$ for some $q>0$).

\subsection{The main results}
The main result of this paper are Theorems \ref{thm:abspolymix} and \ref{thm:absexpmix} below, which are stated in a much more general
setting than the one described in this introduction. Referring to the question raised in Section \ref{sub:funprob} above, they can be 
phrased (in a slightly informal but effective way) as follows.  Here, we assume that the autonomous velocity field satisfies $\uu\in W^{1,\infty}$,
although time-dependence is also allowed.

\medskip

\noindent \textbf{Polynomial mixing} [Theorem \ref{thm:abspolymix}].
Assume that in \eqref{eq:mixing} we have $\varrho(t)\sim t^{-p}$, for some $p>0$. Then there exists a constant $c_0>0$ such that
\begin{align}
\|f^\nu(t)\|_{L^2}\leq \e^{-c_0\nu^{q} t} \|f^{in}\|_{L^2}, \qquad \forall t > \frac{1}{\nu^{q}}, \qquad \text{with}\quad q=\frac{2}{2+p}.
\end{align}
In particular, $\uu$ is relaxation enhancing with time-scale $O(\nu^{-q})$.

\medskip

\noindent \textbf{Exponential mixing} [Theorem \ref{thm:absexpmix}].
Assume that in \eqref{eq:mixing} we have $\varrho(t)\sim \e^{-t^p}$, for some $p>0$. Then there exists a constant $c_0>0$ such that
\begin{align}
\|f^\nu(t)\|_{L^2}\leq \e^{-c_0|\ln\nu|^{-q} t} \|f^{in}\|_{L^2}, \qquad \forall t > |\ln\nu|^{q}, \qquad \text{with}\quad q=\frac{2}{p}.
\end{align}
In particular,  $\uu$ is relaxation enhancing with time-scale $O(|\ln\nu|^{q})$.

\begin{remark}
From the mathematical viewpoint, there is nothing special about the use of the $\H^{-1}$ norm in the estimate \eqref{eq:mixing}.
Indeed, decay of any negative Sobolev estimate implies weak convergence to 0 of the solution $f(t)$. As expected, all our 
results hold in the case an $\H^s\mapsto \H^{-s}$ mixing estimate, with a modification of the enhanced dissipation time-scale 
(cf. Corollary \ref{cor:mixHs}).
\end{remark}

As far as concrete examples are concerned, we then deduce the following new results.

\medskip

\noindent \textbf{Contact Anosov flows} [Corollary \ref{cor:anos}]. 
All contact Anosov flows  on a smooth $2d+1$ dimensional connected compact Riemannian manifold 
are relaxation enhancing with time-scale $O(|\ln\nu|^{2})$.

\begin{remark}
For this result, we heavily exploit the exponential mixing estimate implied by the results of \cite{Liv04}. We also note that the geodesic flow
on any negatively curved space is an example of a contact Anosov flow.
\end{remark}

\begin{remark}
To the best of our knowledge, this is the first example of a flow that induces an enhanced dissipation time-scale that is
faster than $O(\nu^{-1/3})$.
\end{remark}

\medskip

\noindent \textbf{Shear flows with fractional diffusion} [Proposition \ref{prop:enhancedshear}].
Assume that a passive scalar $f^\nu$ is advected by a shear flow $\uu=(u(y),0)$,
where $u$ has a finite number of critical points, and undergoes diffusion given by the fractional laplacian $(-\Delta)^{\gamma/2}$, for
some $\gamma\in(0,2)$. Then $\uu$ is relaxation enhancing (relative to $(-\Delta)^{\gamma/2}$) at a time-scale which depends
on $\gamma$ and on the flatness of the critical points of $u$.

\begin{remark}
This is an extension of the results of \cite{BCZ15} to the case of fractional diffusion. However, the methods here are completely different.
In particular, the hypocoercivity scheme adopted in \cite{BCZ15} seems to be difficult to apply here, due to the nonlocal nature of the diffusion
and the consequent complicated nature of the commutators with the advection term.
\end{remark}

\noindent \textbf{Spiral flows} [Proposition \ref{sub:spiralmix} and Theorem \ref{thm:enhancespiral}].
For $\alpha\geq1$, consider the a passive scalar advected by $\uu(r,\theta)=r^{1+\alpha}(-\sin\theta, \cos\theta)$ in 
the unit disk, where $(r,\theta)$ are polar coordinates. Then $\uu$ is mixing with a rate only depending on $\alpha$, and therefore
relaxation enhancing at a time-scale $O(\nu^{-q_\alpha})$, with $q_\alpha\in(0,1)$, depending explicitly on $\alpha$. Note that by using
the special structure of this flow, the exponent $q_\alpha$ is in fact smaller (hence better) compared to the one given directly by Theorem \ref{thm:abspolymix}.

\begin{remark}
The case $\alpha=1$, from the point of view of mixing rates, has been analyzed in detail in \cite{CLS17}, using different methods. 
\end{remark}

\subsection{Structure of the article}
\ Section \ref{sec:abstract}  sets up the general scheme that allows us to treat a variety of problems of physical interest,
 listed in Section \ref{sub:examples}. We  
 prove the main results on polynomial and exponential mixing in Sections \ref{sub:poly} and \ref{sub:expo}, respectively.
Section \ref{sec:passivescalars} is dedicated to two applications: in Section \ref{sub:Anosov} we prove that all contact
Anosov flows have logarithmic enhanced dissipation time-scale, while in Section \ref{sub:fracto} we analyze the case of shear flows with fractional
diffusion. Finally, Section \ref{sec:spirale} deals with the radial setting of spiral flow, both from the inviscid mixing (cf. Section \ref{sub:spiralmix}) and
the enhanced dissipation (cf. Section \ref{sub:spiralenhanced}) viewpoints.

\section{The abstract result}\label{sec:abstract}
In this section, we rephrase problems \eqref{eq:inviscidpass} and \eqref{eq:viscouspass} in an abstract way.
We only make use of two fundamental properties: the antisymmetric structure of the advection term with respect
to a scalar product, and the dissipative nature of the Laplace operator.

\subsection{The functional setting}\label{sec:functional}
Let $H$ be a separable (real or complex) Hilbert space, endowed with norm
$\|\cdot\|_H$ and scalar (or Hermitian) product $\l\cdot,\cdot\r$, and let $A$ be a strictly positive self-adjoint unbounded linear operator
\begin{align}
A:\mathcal{D}(A)\to H
\end{align}
such that the domain $\mathcal{D}(A)$ is compactly embedded in $H$.
In this way, $A$ possesses a
strictly positive sequence of eigenvalues $\{\lambda_j\}_{j \in \N}$ such that
\begin{align}
\begin{cases}
0<\lambda_1\leq \lambda_2 \leq \ldots,\\
\lambda_j\to \infty \quad\text{for} \quad j \to \infty,
\end{cases}
\end{align}
and associated eigenvectors $\{e_j\}_{j\in\N}$ which form an orthonormal basis for $H$.
Any element $\phi \in H$ can be therefore written as 
\begin{align}
\phi=\sum_{j\geq 1} \phi_j e_j, \qquad \phi_j=\l \phi,e_j\r.
\end{align}
For any $s\in \R$, we can then define the scale of Hilbert spaces $H^s$, 
with norm 
\begin{align}
\|\phi\|^2_{H^s}=\sum_{j\geq 1} \lambda_j^{s}|\phi_j|^2.
\end{align}
Indicating by $P_{\leq R}$  the projection onto the span of the first elements of this basis corresponding to 
eigenvalues $|\lambda|\leq R$,
we deduce the Poincar\'e-like inequalities
\begin{align}\label{eq:lowfreq0}
\| P_{\leq R}\phi\|^2_{H}\leq R^{s}\|\phi\|^2_{H^{-s}}, \qquad s\geq0,
\end{align}
and
\begin{align}\label{eq:highfreq0}
R^{s}\| (I-P_{\leq R})\phi\|^2_{H}\leq \|\phi\|^2_{H^s}, \qquad s\geq0.
\end{align} 
Let $B$ be  an unbounded antisymmetric operator on $H$, such that
there exists a constant $c_B >0$ with the property that
\begin{align}\label{eq:Bdomain}
\| B\phi\|_H\leq c_B\|\phi\|_{H^s}, \qquad \forall \phi\in H^s
\end{align} 
for some $s> 0$, and
\begin{align}\label{eq:Bassumption}
|\Re\l B\phi, A\phi \r|\leq c_B\|\phi\|_{H^1}^2, \qquad \forall \phi\in H^1.
\end{align} 
As it will be clear from the examples in Section \ref{sub:examples}, in the case of passive scalar
these assumption are very much related to a divergence-free assumption on the flow and its regularity.

For $\nu\in (0,1)$ and $t>0$, we study the decay properties of solutions of the \textit{viscous flow}
\begin{align}\label{eq:viscousabs}
\begin{cases}
\de_t f^\nu+ Bf^\nu+\nu A f^\nu=0,\\
f^\nu(0)=f^{in},
\end{cases}
\end{align}
leaning on the mixing properties of the \textit{inviscid flow} 
\begin{align}\label{eq:inviscidabs}
\begin{cases}
\de_tf+Bf=0,\\
 f(0)=f^{in},
\end{cases}
\end{align}
where $f(t)=\e^{Bt}f^{in}$ is generated by the unitary evolution $\{\e^{Bt}\}_{t\in\R}$.
We first recall a few basic facts about the above linear problems.

\subsection{Basic properties of solutions}\label{sub:basicprop}
It is a classical result that for any $f^{in}\in H^1$,  problems \eqref{eq:viscousabs} and \eqref{eq:inviscidabs} admit unique
global in time solutions. Precisely,
\begin{align}
f^\nu\in L^2_{loc}(0,\infty; H^2)\cap C([0,\infty);H^1),
\end{align}
and
\begin{align}
f\in C([0,\infty);H^1).
\end{align}
If $f^{in}\in H$, then \eqref{eq:viscousabs} still has a unique global solution, with regularity
\begin{align}
f^\nu\in L^2_{loc}(0,\infty; H^1)\cap C([0,\infty);H).
\end{align}
In light of the antisymmetry of $B$, we have for the inviscid flow \eqref{eq:inviscidabs} the \textit{energy conservation} 
\begin{align}\label{eq:consener}
\|f(t)\|_{H}=\|f^{in}\|_{H}, \qquad \forall t\geq 0.
\end{align}
In a similar fashion, for the viscous flow \eqref{eq:viscousabs} we have the  \emph{energy equation} 
\begin{align}\label{eq:visc_ode}
\ddt \|f^{\nu}\|^2_{H}+2\nu \| f^\nu\|^2_{H^1}=0.
\end{align}
By multiplying the viscous flow \eqref{eq:viscousabs} by $A f$  we obtain
\begin{align}\label{eq:visc_ode2}
\ddt \|f^{\nu}\|^2_{H^1}+2\nu \| f^\nu\|^2_{H^2}=-2\Re\l B f^\nu, A f^\nu \r.
\end{align}
In view of the definition of $c_B$ \eqref{eq:Bassumption}, it is not hard to check that
\begin{align}\label{eq:visc_ode3}
\ddt \|f^{\nu}\|^2_{H^1}+2\nu \| f^\nu\|^2_{H^2}\leq 2c_B\|f^{\nu}\|_{H^1}^2.
\end{align}
Integrating \eqref{eq:visc_ode3} in time and re-arranging we obtain a bound on the integral of the $H^2$ norm as
\begin{align}\label{eq:intest2}
2\nu \int_0^t\| f^\nu(s)\|^2_{H^2}\dd s\leq 2c_B \int_0^t\|f^\nu(s)\|_{H^1}^2\dd s+\|f^{in}\|^2_{H^1}.
\end{align}
Lastly, by taking the difference between viscous flow \eqref{eq:viscousabs} and inviscid flow \eqref{eq:inviscidabs} and using the antisymmetry
of $B$ once more, 
we find that
\begin{align}
\ddt \| f^\nu-f\|^2_{H}=-2\nu \Re\l A f^\nu, f^{\nu}-f\r_{H}
\end{align}
From this and energy conservation \eqref{eq:consener}, we can deduce 
\begin{align}
\ddt \| f^\nu-f\|^2_{H}\leq2\nu \|f^\nu\|_{H^2} \|f\|_{H}.
\end{align}
Hence, for any $\tau_0\geq 0$, we have
\begin{align}\label{eq:diff21}
\| f^\nu(t+\tau_0)-f(t+\tau_0)\|^2_{H}&\leq\| f^\nu(\tau_0)-f(\tau_0)\|^2_{H}+ 2\nu \| f(\tau_0)\|_{H} \int_{\tau_0}^{\tau_0+t}\|f^\nu(s)\|_{H^2}\dd s\notag\\
&\leq\| f^\nu(\tau_0)-f(\tau_0)\|^2_{H}+ \sqrt{2\nu t}\left(2\nu\int_{\tau_0}^{\tau_0+t}\|f^\nu(s)\|^2_{H^2}\dd s\right)^{1/2}\| f(\tau_0)\|_{H}.
\end{align}
Combining it with the bound on the integral of the $H^2$ norm \eqref{eq:intest2}, we get an estimate on the proximity of solutions to the viscous and inviscid flow
\begin{align}\label{eq:close1}
 \| f^\nu(t)-f(t)\|^2_{H}
&\leq\| f^\nu(\tau_0)-f(\tau_0)\|^2_{H}\notag\\
&\quad+ \sqrt{ t}\left(4c_B \nu\int_{\tau_0}^{\tau_0+t}\|f^\nu(s)\|_{H^1}^2\dd s+2\nu\|f^{\nu}(\tau_0)\|^2_{H^1}\right)^{1/2}\| f(\tau_0)\|_{H}.
\end{align}

\subsection{Non-autonomous problems}
An important generalization of the above setting is achieved when considering a family $\{B(t)\}_{t\geq 0}$ 
of unbounded antisymmetric operators on $H$. All the results of this article hold in this case,
provided that it is possible to deduce analogous properties to those derived in the above Section \ref{sub:basicprop} for
the solutions of the non-autonomous viscous problem
\begin{align}\label{eq:viscousabsNON}
\begin{cases}
\de_t f^\nu+ B(t)f^\nu+\nu A f^\nu=0,\\
f^\nu(0)=f^{in},
\end{cases}
\end{align}
and its inviscid version
\begin{align}\label{eq:inviscidabsNON}
\begin{cases}
\de_tf+B(t)f=0,\\
 f(\tau_0)=f^{\tau_0}.
\end{cases}
\end{align}
where $f^{\tau_0}$ is the initial condition assigned at an arbitrary initial time $\tau_0\geq0$.
A particular case, yet very important, is when $B(t)=\uu(t)\cdot\nabla$, for a Lipschitz continuous divergence-free 
velocity field $\uu$ that is uniformly  bounded in time. This is discussed in detail in Section \ref{rem:concrete} below, where we shall see
that, in an appropriate function space setting, properties \eqref{eq:Bdomain}-\eqref{eq:Bassumption} hold uniformly in time. Hence,
\eqref{eq:consener} and \eqref{eq:visc_ode} hold in the same fashion, as well as \eqref{eq:visc_ode3}.

\subsection{Polynomial mixing}\label{sub:poly}
The goal is to provide a link between decay (mixing) properties of the inviscid flow \eqref{eq:inviscidabsNON}, which solely 
depend on the structure of the operator $B$, and the creation of time-scales for the viscous flow \eqref{eq:viscousabsNON}
which are faster than the purely diffusive one, proportional to $1/\nu$. 
Our first result is dealing with the case in which the inviscid problem undergoes mixing at polynomial rates. 

\begin{theorem}[Polynomial mixing]\label{thm:abspolymix}
Under the assumption \eqref{eq:Bassumption}, assume that solutions to \eqref{eq:inviscidabsNON} satisfy the mixing estimate
\begin{align}\label{eq:inviscid_damping}
\| f(t)\|_{H^{-1}} \leq \frac{a}{(t-\tau_0)^p}\| f^{\tau_0}\|_{H^1}, \qquad \forall t>\tau_0,\ \forall f^{\tau_0}\in H^1,
\end{align}
for an arbitrary initial time $\tau_0\geq 0$,
for some $p\in(0,\infty)$ and some $a>0$.
Then, for every $f^{in}\in H$ there holds the estimate 
\begin{align}
    \|f^{\nu}(t)\|_H\le \e^{-c_0\nu^q t}\|f^{in}\|_H,\qquad \forall t>\frac{1}{\nu^{q}},
\end{align}
with 
\begin{align}
q=\frac{2}{2+p},\qquad c_0=\frac{1}{128}\min\left\{\frac{1}{2(1+c_B)},\frac{1}{a 4^p}\right\},
\end{align}
where $c_B$ is given by \eqref{eq:Bassumption}.
\end{theorem}

\begin{remark}
If the problem under consideration were autonomous, namely if $B(t)\equiv B$ for all $t\geq 0$, estimate \eqref{eq:inviscid_damping}
could be assumed for $\tau_0=0$ only. 
\end{remark}

We preliminary note that from the inviscid damping estimate \eqref{eq:inviscid_damping}
combined with the estimate for low frequencies \eqref{eq:lowfreq0}, for every $R\geq 0$ we obtain
\begin{align}\label{eq:inviscid_damping2}
\| P_{\leq R}f(t)\|_{H}\leq \frac{a R}{(t-\tau_0)^{p}}\|f^{\tau_0}\|_{H^1}, \qquad \forall t> \tau_0.
\end{align}
This estimate is clearly indicative of the decaying behavior of the low frequencies for the inviscid problem,
and therefore precisely describes the transfer of energy towards higher frequencies due to mixing. The main
point in the proof is to combine this effect with dissipation.

\begin{proof}[Proof of Theorem \ref{thm:abspolymix}]
 We first show that for all $\nu<1$ and every $\tau_\star\geq0$, we have the inequality 
\begin{align}\label{decay delta1}
\nu\int_{\tau_\star}^{\tau_\star+\nu^{-q}}\|f^\nu(t)\|_{H^1}^2\dd t\geq \delta\|f^{\nu}(\tau_\star)\|^2_{H},
\end{align}
where
\begin{align}\label{choice of delta1}
    \delta=\frac{1}{64}\min\left\{\frac{1}{2(1+c_B)},\frac{1}{a 4^{p}}\right\}.
\end{align}
Without loss of generality, we can assume by linearity that $\|f^{\nu}(\tau_\star)\|_{H}=1$. 
For the sake of simplicity, we will show \eqref{decay delta1} for $\tau_\star=0$. It will be clear from the proof that this choice is irrelevant.
Towards a contradiction, we assume that
\begin{align}\label{contradictiohyp1}
  \nu\int_0^{\nu^{-q}}\|f^\nu(t)\|_{H^1}^2\dd t<\delta.  
\end{align}
We first claim that there exists a $\tau_1\in [0,\nu^{-q}]$ so that 
\begin{align}\label{eq:deftau11}
\nu\int_{\tau_1}^{\tau_1+\nu^{-q/2}}\|f^\nu(s)\|_{H^1}^2\dd s< 2\delta \nu^{q/2}.
\end{align}
We consider $t_i=(i-1)\nu^{-q/2}$ and  $\lfloor\nu^{-q/2}\rfloor$ the largest integer smaller than $\nu^{-q/2}$. Then
\begin{align}
    \lfloor\nu^{-q/2}\rfloor \min_{i} \nu\int_{t_i}^{t_{i+1}}\|f^\nu(s)\|_{H^1}^2\dd s\le \sum_{i=1}^{\lfloor\nu^{-q/2}\rfloor} \nu\int_{t_i}^{t_{i+1}}\|f^\nu(s)\|_{H^1}^2\dd s\le \nu \int_0^{\nu^{-q}}\|f^\nu(s)\|_{H^1}^2\dd s<\delta.
\end{align}
Using that $\nu<1$ we have that $\lfloor\nu^{-q/2}\rfloor\ge \nu^{-q/2}/2$, and the existence of $\tau_1$ follows. Moreover, by Chebyshev's inequality and \eqref{eq:deftau11} we can find $\tau_0\in [\tau_1,\tau_1+\nu^{-q/2}/2]$ such that 
\begin{align}\label{reversepoicareattau01}
\nu\|f^\nu(\tau_0)\|_{H^1}^2< 4\delta\nu^q.    
\end{align}
Moreover, 
\begin{align}\label{boundtau01}
    \nu\int_{\tau_0}^{\tau_0+\nu^{-q/2}/2}\|f^\nu(s)\|_{H^1}^2\dd s< 2\delta \nu^{q/2},
\end{align}
by \eqref{eq:deftau11}.
Now we take $f^{\tau_0}=f^\nu(\tau_0)$ as initial datum for the inviscid problem \eqref{eq:inviscidabsNON} with initial time $\tau_0$ and denote the solution by $f(t+\tau_0)$ with $t\geq 0$. Using that $\|f^\nu(\tau_0)\|_H\leq 1$, the estimate on the proximity of the two flows \eqref{eq:close1}, the properties of $\tau_0$ \eqref{reversepoicareattau01} and \eqref{boundtau01}, we have
\begin{align}\label{eq:close2_1}
\| f^\nu(\tau_0+t)-f(\tau_0+t)\|^2_{H} 
&\leq \sqrt{ t}\left(4c_B \nu\int_{\tau_0}^{\tau_0+t}\|f^\nu(s)\|_{H^1}^2\dd s+2\nu\|f^\nu(\tau_0)\|^2_{H^1}\right)^{1/2}\notag\\
&\leq \sqrt{\frac{\nu^{-q/2}}{2}}\left(8c_B\delta\nu^{q/2} +8\delta\nu^q\right)^{1/2}\leq\frac{1}{4},
\end{align}
for all $t\in [0,\frac12\nu^{-q/2}]$, since $\delta\leq\frac{1}{128(1+c_B)}$ and $\nu<1$. 
By energy conservation \eqref{eq:consener} for the inviscid problem, the energy dissipation for the viscous evolution \eqref{eq:visc_ode}, the contradiction hypothesis \eqref{contradictiohyp1} and the choice of $\delta$ \eqref{choice of delta1}, we have 
\begin{align}\label{initial energy}
\|f(t+\tau_0)\|^2_H=\|f^\nu(\tau_0)\|^2_{H}=\|f^{in}\|^2_{H}-2\nu\int_0^{\tau_0}\|f^{\nu}(s)\|_{H^1}^2\dd s\geq 1-2\delta\geq\frac{3}{4},    
\end{align}
for all $t\in [0,\frac{1}{2}\nu^{-q/2}]$. 
Now, using the mixing estimate  \eqref{eq:inviscid_damping2} together with \eqref{reversepoicareattau01}, we obtain  for any $R\geq 1$ and any $t\in [\frac{1}{4}\nu^{-q/2},\frac{1}{2}\nu^{-q/2}]$ that
\begin{align}\label{eq:mixlow}
\| P_{\leq R}f(\tau_0+t)\|^2_{H}
\leq\frac{a^2 R}{t^{2p}}\|f^\nu(\tau_0)\|^2_{H^1} \leq  4^{2p+1}a^2 \delta\nu^{q(p+1)-1}R.
\end{align}
Using the energy conservation for the inviscid evolution \eqref{eq:consener} and the initial energy bound \eqref{initial energy} we have
\begin{align}
\| (I-P_{\leq R})f(\tau_0+t)\|^2_{H}=\| f(\tau_0+t)\|^2_{H}-\| P_{\leq R}f(\tau_0+t)\|^2_{H}\geq \frac{3}{4}-\| P_{\leq R}f(\tau_0+t)\|^2_{H}.
\end{align}
We deduce that
\begin{align}\label{eq:highfreq1_11}
\| (I-P_{\leq R})f(\tau_0+t)\|^2_{H}\geq\frac{3}{4}-  4^{2p+1}a^2 \delta\nu^{q(p+1)-1}R, \qquad \forall t\in \left[\frac{1}{4}\nu^{-q/2},\frac{1}{2}\nu^{-q/2}\right] .
\end{align}
Now, appealing to \eqref{eq:close2_1}, we find
\begin{align}
\frac{3}{4}-  4^{2p+1}a^2 R\delta\nu^{q(p+1)-1}
&\leq \| (I-P_{\leq R})f(\tau_0+t)\|^2_{H}\notag\\
&\leq 2\| (I-P_{\leq R})(f^\nu(\tau_0+t)-f(\tau_0+t))\|^2_{H}+2\| (I-P_{\leq R})f^\nu(\tau_0+t)\|^2_{H}\notag\\
&\leq 2\| f^\nu(\tau_0+t)-f(\tau_0+t)\|^2_{H}+2\| (I-P_{\leq R})f^\nu(\tau_0+t)\|^2_{H}\notag\\
&\leq \frac12+2\| (I-P_{\leq R})f^\nu(\tau_0+t)\|^2_{H},
\end{align}
so we deduce that 
\begin{align}\label{eq:highfreq2_11}
\| (I-P_{\leq R})f^\nu(\tau_0+t)\|^2_{H}\geq \frac{1}{8}\left(1 -4^{2p+2}a^2 \delta\nu^{q(p+1)-1}R\right), \qquad \forall t\in \left[\frac{1}{4}\nu^{-q/2},\frac{1}{2}\nu^{-q/2} \right].
\end{align}
In turn, from the estimate for the high frequency \eqref{eq:highfreq0} we have that
\begin{align}\label{eq:highfreq3_11}
\| f^\nu(\tau_0+t)\|^2_{H^1}&\geq \frac{R}{8}\left(1 -4^{2p+2}a^2 \delta\nu^{q(p+1)-1}R\right), \qquad \forall t\in \left[\frac{1}{4}\nu^{-q/2},\frac{1}{2}\nu^{-q/2} \right].
\end{align}
The left-hand side is independent of $R$, while the right-hand side is a quadratic function of $R$, so we can maximize the right-hand side with respect to $R$. To this end, we pick
\begin{align}
R= \frac{1}{2}\frac{1}{4^{2p+2}a^2 \delta\nu^{q(p+1)-1}},
\end{align}
and obtain
\begin{align}
\| f^\nu(\tau_0+t)\|^2_{H^1}\geq\frac{1}{512}\frac{1}{4^{2p}a^2 \delta\nu^{q(p+1)-1}}, \qquad \forall t\in \left[\frac{1}{4}\nu^{-q/2},\frac{1}{2}\nu^{-q/2}\right]\label{eq:finalestforH1}.
\end{align}
Integrating over $\left(\frac{1}{4}\nu^{-q/2},\frac{1}{2}\nu^{-q/2}\right)$ and using the bound \eqref{boundtau01}, we obtain
\begin{align}
2\delta\nu^{q/2}\ge \nu \int_{\frac{1}{4}\nu^{-q/2}}^{\frac{1}{2}\nu^{-q/2}}\| f^\nu(\tau_0+t)\|^2_{H^1}\dd t >\frac{1}{2048}\frac{\nu^{-q/2}}{4^{2p}a^2 \delta\nu^{q(p+1)-2}}.
\end{align}
By re-arranging and recalling that $q(p+2)-2=0$, we get  that 
\begin{align}
\delta^2>\frac{1}{4096}\frac{1}{4^{2p} a^2}
\end{align}
which contradicts our choice of $\delta$ \eqref{choice of delta1} and proves the desired estimate \eqref{decay delta1}.

To show the exponential decay we iterate \eqref{decay delta1}, use that the energy is not increasing in time and the time-integrated
version of \eqref{eq:visc_ode} to obtain  
\begin{align}\label{aux final}
    \|f^{\nu}(t)\|_H^2\le \|f^{\nu}(\lfloor\nu^q t\rfloor\nu^{-q})\|_H^2\le (1-2\delta)^{\lfloor\nu^q t\rfloor}\|f^{in}\|_H^2\le \e^{-2\delta \lfloor\nu^q t\rfloor}\|f^{in}\|_H^2,
\end{align}
where we have denoted by $\lfloor\nu^q t\rfloor$ the largest integer smaller than $\nu^q t$ and we have used the convexity of the exponential. Finally, if $\nu^{-q}< t$, then 
$$
\lfloor\nu^q t\rfloor>\frac{t\nu^q}{2},
$$ 
which combined with \eqref{aux final} yields the desired estimate. The proof is over.
\end{proof}

In fact, we can modify the mixing estimate \eqref{eq:inviscid_damping} to a general $H^{-s}$ norm in a simple way.

\begin{corollary}[General mixing estimates]\label{cor:mixHs}
Under the assumption \eqref{eq:Bassumption}, assume that solutions to \eqref{eq:inviscidabsNON} satisfy the mixing estimate
\begin{align}\label{eq:inviscid_dampings}
\| f(t)\|_{H^{-s}} \leq \frac{a}{(t-\tau_0)^p}\| f^{\tau_0}\|_{H^s}, \qquad \forall t\geq\tau_0,\ \forall f^{\tau_0}\in H^1,
\end{align}
for an arbitrary initial time $\tau_0\geq 0$, for some $p\in(0,\infty)$ and some $a,s>0$.
Then, for every $f^{in}\in H$ there holds the estimate 
\begin{align}
    \|f^{\nu}(t)\|_H\le \e^{-c_s\nu^{q_s} t}\|f^{in}\|_H,\qquad \forall t>\frac{1}{\nu^{q_s}},
\end{align}
with 
\begin{align}
q_s=\frac{\max\{1,s\}+s}{\max\{1,s\}+s+p},
\end{align}
and
\begin{align}
c_s=\frac12\min\left\{\frac{1}{128(1+c_B)},\left(\frac{s}{16(s+1)}\right)^\frac{s}{1+s}\left(\frac{\lambda_1^{1-s}}{4^{2p+2}(s+1)a^2 }\right)^\frac{1}{1+s}\right\},
\end{align}
where $c_B$ is given by \eqref{eq:Bassumption}.
\end{corollary}

\begin{proof}
Let us first treat the case $s\in(0,1]$. In this case, the choice of $\delta$ in \eqref{choice of delta1} has to be modified to
\begin{align}\label{choice of delta1s}
    \delta_s= \min\left\{\frac{1}{128(1+c_B)},\left(\frac{s}{16(s+1)}\right)^\frac{s}{1+s}\left(\frac{\lambda_1^{1-s}}{4^{2p+2}(s+1)a^2 \delta_s\nu^{q(p+1)-1}}\right)^\frac{1}{1+s}\right\}.
\end{align}
Using \eqref{eq:lowfreq0}, \eqref{eq:inviscid_dampings} and the Poincar\'e inequality (thanks to $s\leq 1$),  we would replace \eqref{eq:mixlow} with
\begin{align}\label{eq:mixlows}
\| P_{\leq R}f(\tau_0+t)\|^2_{H}\leq  \frac{a^2 R^s}{t^{2p}}\|f^\nu(\tau_0)\|^2_{H^s}
\leq\frac{a^2 R^s}{t^{2p}\lambda_1^{1-s}}\|f^\nu(\tau_0)\|^2_{H^1} \leq  \frac{4^{2p+1}a^2 \delta_s\nu^{q(p+1)-1}}{\lambda_1^{1-s}}R^s.
\end{align}
In turn, from the analogous of \eqref{eq:highfreq3_11}, namely
\begin{align}\label{eq:highfreq3_1s}
\| f^\nu(\tau_0+t)\|^2_{H^1}&\geq \frac{R}{8}\left(1 -\frac{4^{2p+2}a^2 \delta_s\nu^{q(p+1)-1}}{\lambda_1^{1-s}}R^s\right), \qquad \forall t\in \left[\frac{1}{4}\nu^{-q/2},\frac{1}{2}\nu^{-q/2} \right].
\end{align}
we deduce by optimizing in $R$ that
\begin{align}
\| f^\nu(\tau_0+t)\|^2_{H^1}\geq\frac{s}{8(s+1)}\left(\frac{\lambda_1^{1-s}}{4^{2p+2}(s+1)a^2 \delta_s\nu^{q(p+1)-1}}\right)^{1/s}, \qquad \forall t\in \left[\frac{1}{4}\nu^{-q/2},\frac{1}{2}\nu^{-q/2}\right]\label{eq:finalestforH1s}.
\end{align}
The proof then follows word for word, with the condition that
\begin{align}
q(1+s+p)-1-s=0,
\end{align}
as claimed.

The proof of the case $s>1$ is even simpler. Indeed, since $\|f(t)\|_{L^2}=\|f^{\tau_0}\|_{L^2}$, we can use standard interpolation theory (see \cite{Tartar}*{Lemma 22.3}) and obtain from \eqref{eq:inviscid_dampings} that
\begin{align}
\| f(t)\|_{H^{-1}} \leq \frac{a^{1/s}}{(t-\tau_0)^{p/s}}\| f^{\tau_0}\|_{H^1}, \qquad \forall t\geq\tau_0,\ \forall f^{\tau_0}\in H^1.
\end{align}
Thus the result in this case follows by a direct application of Theorem \ref{thm:abspolymix}. The proof is over.
\end{proof}

\begin{remark}
Generically speaking, the parameters $s$ and $p$ in \eqref{eq:inviscid_dampings} are not independent of each other.
In particular, it is natural to expect that $p\to0$ as $s\to 0$. 
\end{remark}

\subsection{Exponential mixing}\label{sub:expo}
Regarding exponential mixing, we prove a result which parallels that of Theorem \ref{thm:abspolymix}, obtaining
logarithmic rates of decay. 

\begin{theorem}[Exponential mixing]\label{thm:absexpmix}
Under the assumption \eqref{eq:Bassumption}, assume that solutions to \eqref{eq:inviscidabsNON} satisfy the mixing estimate 
\begin{align}\label{eq:inviscid_dampingexp}
\| f(t)\|_{H^{-1}} \leq a_1 \e^{-a_2 (t-\tau_0)^p}\| f^{\tau_0}\|_{H^1}, \qquad \forall t\geq\tau_0,\ \forall f^{\tau_0}\in H^1,
\end{align}
for an arbitrary initial time $\tau_0\geq 0$,
for some $a_1,a_2, p>0$. Then, for every $f^{in}\in H$ and every
\begin{align}\label{eq:ucsadf}
0<\nu < \min\left\{\e^{-\frac{4^p}{2 a_2}},\e^{-1},\e^{-a_1^{p/2}}\right\},
\end{align}
there holds the estimate 
\begin{align}
    \|f^{\nu}(t)\|_H\le \e^{-c_0|\ln \nu|^{-2/p} t}\|f^{in}\|_H,\qquad \forall t>|\ln \nu|^{2/p}, 
\end{align}
with 
\begin{align}
    c_0=\frac{1}{128}\min\left\{\frac{a_2^{2/p}}{32(1+c_B)},1\right\}
\end{align}
and $c_B$ given by \eqref{eq:Bassumption}.
\end{theorem}

\begin{proof}[Proof of Theorem \ref{thm:absexpmix}]
As in the proof of Theorem \ref{thm:abspolymix}, the main point is to show that
that for all $f^{in}\in L^2$ and $\nu$ complying with \eqref{eq:ucsadf}, we have the inequality 
\begin{align}\label{decay delta}
\nu\int_0^{|\ln \nu|^{2/p}}\|f^\nu(t)\|_{H^1}^2\dd t \geq \delta\|f^{in}\|^2_{H},
\end{align}
where
\begin{align}\label{choice of delta}
    \delta=\frac{1}{64}\min\left\{\frac{a_2^{2/p}}{32(1+c_B)},\frac{1}{a_1}\right\}.
\end{align}
Without loss of generality, we can assume by linearity that $\|f^{in}\|_{H}=1$. Towards a contradiction, we assume that
\begin{align}\label{contradictiohyp}
  \nu\int_0^{|\ln \nu|^{2/p}}\|f^\nu(t)\|_{H^1}^2 \dd t<\delta.  
\end{align}
We show that there exists a $\tau_1\in [0,|\ln \nu|^{2/p}]$ so that 
\begin{align}\label{eq:deftau1}
\nu\int_{\tau_1}^{\tau_1+c|\ln \nu|^{1/p}}\|f^\nu(s)\|_{H^1}^2\dd s< \frac{2c}{|\ln \nu|^{1/p}}\delta ,
\end{align}
where $c=4(2a_2)^{-1/p}$.
We consider $t_i=(i-1)c|\ln \nu|^{1/p}$ and  $\lfloor c^{-1}|\ln \nu|^{1/p}\rfloor$ the largest integer smaller than $c^{-1}|\ln \nu|^{1/p}$. Then
\begin{align}
    \lfloor c^{-1}|\ln \nu|^{1/p}\rfloor \min_{i} \nu\int_{t_i}^{t_{i+1}}\|f^\nu(s)\|_{H^1}^2\dd s&\le \sum_{i=1}^{\lfloor c^{-1}|\ln \nu|^{1/p}\rfloor} \nu\int_{t_i}^{t_{i+1}}\|f^\nu(s)\|_{H^1}^2\dd s\\
    &\le \nu \int_0^{|\ln \nu|^{2/p}}\|f^\nu(s)\|_{H^1}^2\dd s<\delta.
\end{align}
Using that $\nu<\e^{-c^p}$ we have that $\lfloor c^{-1}|\ln \nu|^{1/p}\rfloor\ge c^{-1} |\ln \nu|^{1/p}/2$, and the existence of $\tau_1$ follows. Moreover, by Chebyshev's inequality and \eqref{eq:deftau1} we can find $\tau_0\in [\tau_1,\tau_1+\frac{c}{2}|\ln \nu|^{1/p}]$ such that 
\begin{align}\label{reversepoicareattau0}
\nu\|f^\nu(\tau_0)\|_{H^1}^2< 4c\delta|\ln \nu|^{-2/p},  
\end{align}
and 
\begin{align}\label{boundtau0}
    \nu\int_{\tau_0}^{\tau_0+\frac{1}{2}|\ln \nu|^{1/p}}\|f^\nu(s)\|_{H^1}^2\dd s< 2c\delta |\ln \nu|^{-1/p},
\end{align}
by \eqref{eq:deftau1}.

Now we take $f^\nu(\tau_0)$ as initial datum for the inviscid problem \eqref{eq:inviscidabsNON} with initial time $\tau_0$ and denote the solution by $f(t+\tau_0)$ with $t\geq 0$. Using the estimate on the proximity of the two flows \eqref{eq:close1}, the properties of $\tau_0$ \eqref{reversepoicareattau0} and \eqref{boundtau0}, we have
\begin{align}\label{eq:close2_1exp}
\| f^\nu(\tau_0+t)-f(\tau_0+t)\|^2_{H} 
&\leq \sqrt{ t}\left(4c_B \nu\int_{\tau_0}^{\tau_0+t}\|f^\nu(s)\|_{H^1}^2\dd s+2\nu\|f^\nu(\tau_0)\|^2_{H^1}\right)^{1/2}\notag\\
&\leq \sqrt{\frac{c}{2}|\ln \nu|^{1/p}}\left(8c_Bc\delta|\ln \nu|^{-1/p} +8c\delta|\ln \nu|^{-2/p}\right)^{1/2}<\frac{1}{4},
\end{align}
for all $t\in [0,\frac{c}{2}|\ln \nu|^{1/p}]$, since $\delta<\frac{1}{128 c^2(1+c_B)}$ and $\nu<\e^{-1}$. 
By energy conservation \eqref{eq:consener} for the inviscid problem, the energy dissipation for the viscous evolution \eqref{eq:visc_ode}, the contradiction hypothesis \eqref{contradictiohyp} and the choice of $\delta$ \eqref{choice of delta}, we have 
\begin{align}\label{initial energyexp}
\|f(t+\tau_0)\|^2_H=\|f^\nu(\tau_0)\|^2_{H}=\|f^{in}\|^2_{H}-2\nu\int_0^{\tau_0}\|f^{\nu}(s)\|_{H^1}^2\dd s\geq 1-2\delta\geq\frac{3}{4},    
\end{align}
for all $t\in [0,\frac{c}{2}|\ln \nu|^{1/p}]$. 
Now, using the mixing estimate  \eqref{eq:inviscid_dampingexp} together with \eqref{reversepoicareattau0}, we obtain  for any $R\geq 1$ and any $t\in [\frac{c}{4}|\ln \nu|^{1/p},\frac{c}{2}|\ln \nu|^{1/p}]$ that
\begin{align}
\| P_{\leq R}f(\tau_0+t)\|^2_{H}\leq  a_1^2\e^{-2a_2t^p} R\|f^\nu(\tau_0)\|^2_{H^1} \leq 4  a_1^2\e^{-\frac{2}{4^p}a_2c^p|\ln \nu|} \delta \nu^{-1}| \ln \nu|^{-2/p}R=4a_1^2\nu \delta|\ln \nu|^{-2/p}R,
\end{align}
where we have used the definition of $c=4a_2^{-1/p}$.
Using the energy conservation for the inviscid evolution \eqref{eq:consener} and the initial energy bound \eqref{initial energyexp} we have
\begin{align}
\| (I-P_{\leq R})f(\tau_0+t)\|^2_{H}=\| f(\tau_0+t)\|^2_{H}-\| P_{\leq R}f(\tau_0+t)\|^2_{H}\geq \frac{3}{4}-\| P_{\leq R}f(\tau_0+t)\|^2_{H}.
\end{align}
We deduce that
\begin{align}\label{eq:highfreq1_1}
\| (I-P_{\leq R})f(\tau_0+t)\|^2_{H}\geq\frac{3}{4}-  4  a_1^2\nu \delta|\ln \nu|^{-2/p}R, \qquad \forall t\in \left[\frac{c}{4}|\ln\nu|^{1/p},\frac{c}{2}|\ln\nu|^{1/p}\right].
\end{align}
Now, appealing to \eqref{eq:close2_1exp}, we find
\begin{align}
\frac{3}{4}-  4  a_1^2\nu R\delta|\ln \nu|^{-2/p}
&\leq \| (I-P_{\leq R})f(\tau_0+t)\|^2_{H}\notag\\
&\leq 2\| (I-P_{\leq R})(f^\nu(\tau_0+t)-f(\tau_0+t))\|^2_{H}+2\| (I-P_{\leq R})f^\nu(\tau_0+t)\|^2_{H}\notag\\
&\leq 2\| f^\nu(\tau_0+t)-f(\tau_0+t)\|^2_{H}+2\| (I-P_{\leq R})f^\nu(\tau_0+t)\|^2_{H}\notag\\
&\leq \frac12+2\| (I-P_{\leq R})f^\nu(\tau_0+t)\|^2_{H},
\end{align}
so we deduce that 
\begin{align}\label{eq:highfreq2_1}
\| (I-P_{\leq R})f^\nu(\tau_0+t)\|^2_{H}\geq \frac{1}{8}\left(1 -16  a_1^2\nu \delta|\ln \nu|^{-2/p}R\right), \qquad \forall t\in \left[\frac{c}{4}|\ln\nu|^{1/p},\frac{c}{2}|\ln\nu|^{1/p} \right].
\end{align}
In turn, from the estimate for the high frequency \eqref{eq:highfreq0} we have that
\begin{align}\label{eq:highfreq3_1}
\| f^\nu(\tau_0+t)\|^2_{H^1}&\geq \frac{R}{8}\left(1 -16  a_1^2\nu \delta|\ln \nu|^{-2/p}R\right), \qquad \forall t\in \left[\frac{c}{4}|\ln\nu|^{1/p},\frac{c}{2}|\ln\nu|^{1/p} \right].
\end{align}
The left-hand side is independent of $R$, while the right-hand side is a quadratic function of $R$, so we can maximize the right-hand side with respect to $R$. To this end, we pick
\begin{align}
R= \frac{1}{2}\frac{1}{16  a_1^2\nu \delta|\ln \nu|^{-2/p}},
\end{align}
and obtain
\begin{align}
\| f^\nu(\tau_0+t)\|^2_{H^1}\geq\frac{1}{512}\frac{1}{a_1^2\nu\delta|\ln \nu|^{-2/p}}, \qquad \forall t\in \left[\frac{c}{4}|\ln\nu|^{1/p},\frac{c}{2}|\ln\nu|^{1/p}\right].
\end{align}
Multiplying by $\nu$, integrating over $\left[\frac{c}{4}|\ln\nu|^{1/p},\frac{c}{2}|\ln\nu|^{1/p}\right]$ and using the bound \eqref{boundtau0}, we obtain
\begin{align}
2c\delta|\ln\nu|^{-1/p}\ge \nu \int_{\frac{c}{4}|\ln\nu|^{1/p}}^{\frac{c}{2}|\ln\nu|^{1/p}}\| f^\nu(\tau_0+t)\|^2_{H^1}\dd t>\frac{1}{2048}\frac{c|\ln\nu|^{1/p}}{a_1^2\delta|\ln \nu|^{-2/p}}.
\end{align}
By re-arranging, we get  that 
\begin{align}
\delta^2>\frac{1}{4096}\frac{1}{a_1^2}|\ln\nu|^{4/p},
\end{align}
which, from the restriction on $\nu$ in \eqref{eq:ucsadf}, contradicts our choice of $\delta$ \eqref{choice of delta} and proves the desired estimate \eqref{decay delta}. The proof is over.
\end{proof}

\section{Concrete examples}\label{sub:examples}

In this section, we give a few examples that fall in the class of problems under consideration. Some of them are treated in detail 
in later sections.

\subsection{Passive scalars on Riemannian manifolds}\label{rem:concrete}
We consider $M$ a compact Riemannian manifold with a metric $g$ and volume form $\omega$. A concrete realization of the above abstract setting is achieved by taking
\begin{align}
H=\left\{\varphi\in L_{\omega}^2(M): \int_{M}\varphi(x)\dd \omega(x)=0\right\}, \qquad A\varphi=(-\Delta_{M})\varphi, \qquad B(t)\varphi=g(\uu(t), \nabla_M \varphi),
\end{align}
for a smooth and divergence-free velocity field $\u(t,x):[0,\infty)\times M\to TM$, such that 
$\u\in L^\infty_t W_x^{1,\infty}$. The Laplace-Beltrami operator $-\Delta_{M}$ satisfies all the assumptions above. 
Regarding the operators $B(t)$, the antisymmetry holds thanks to the divergence-free assumption, which is equivalent to the volume form being invariant under the flow induced by $\u(t)$. 

Next, we check the assumption \eqref{eq:Bassumption}. Integrating by parts we get
\begin{align}
\l g(\u(t), \nabla_M \varphi), -\Delta_M \varphi \r_{H}&=\int_{M}g(\u(t), \nabla_M \varphi)(-\Delta_M\varphi)\,\dd\omega(x)\notag\\
&=\int_{M}g(\nabla_M[g(\u(t), \nabla_M \varphi)],\nabla_M\varphi)\,\dd\omega(x).\label{eq:intbyparts}
\end{align}
By applying the product rule we get
\begin{align}\label{eq:productrule}
\nabla_M[g(\u(t), \nabla_M \varphi)]=\nabla_M\u(t)[\nabla_M \varphi] + D^2_M \varphi[\u(t)].
\end{align}
Using the symmetry of the Hessian, we obtain
\begin{align}\label{eq:symmetry}
g(D^2_M \varphi[\u(t)],\nabla_M\varphi)=D^2_M \varphi[\u(t),\nabla_M\varphi]=g(\u(t), D^2_M \varphi[\nabla_M\varphi]).
\end{align}
We notice that
\begin{align}\label{eq:perfectderivative}
\frac{1}{2}\nabla_M[g(\nabla_M \varphi,\nabla_M \varphi)]=D^2_M \varphi[\nabla_M \varphi],
\end{align}
and that the divergence free condition implies that
\begin{align}\label{eq:divfree}
\int_M g(\u(t), \nabla_M[g(\nabla_M \varphi,\nabla_M \varphi)])\,\dd\omega(x)=0.
\end{align}
Therefore, by combining \eqref{eq:intbyparts}, \eqref{eq:productrule}, \eqref{eq:symmetry}, \eqref{eq:perfectderivative}, \eqref{eq:divfree} and bounding the first derivatives of $\u(t)$ we have
\begin{align}
\l g(\u(t), \nabla_M \varphi), -\Delta_M \varphi \r_{H}\le   \|\u\|_{L_t^\infty W_x^{1,\infty}}\|\varphi \|^2_{H^1},
\end{align}
which is precisely \eqref{eq:Bassumption}.

\subsection{Fractional diffusion}\label{rem:concrete2}
Thanks to the generality of our setting, we can also handle the case of fractional diffusion on the periodic domain $\T^d$, namely
\begin{align}\label{eq:fractpass}
\de_t f^\nu +\u(t) \cdot \nabla f^\nu+\nu\Lambda^\gamma f^\nu=0,
\end{align}
where $\Lambda=\sqrt{-\Delta}$ is the Zygmund operator, $\gamma\in(0,2)$ is a fixed parameter 
measuring the strength of the diffusion, and $\u\in L^\infty_t W_x^{1,\infty}$ is divergence free. For the fractional laplacian, 
we will mainly use its representation as the singular integral 
\begin{align}\label{eq:fraclapc}
\Lambda^\gamma f^\nu(x)= c_\gamma \sum_{k \in \ZZ^d}  \int_{\T^d}   \frac{f^\nu(x) - f^\nu(x+y)}{|y- 2\pi k|^{d+\gamma}} \dd y= c_\gamma \,\mathrm{P.V.}\int_{\R^d} \frac{f^\nu(x)-f^\nu(x+y)}{|y|^{d+\gamma}}\dd y, 
\end{align} 
abusing notation and  denoting by $f^\nu$ also the periodic extension of $f^\nu$ to the whole space.

While the dissipative operator $A=\Lambda^{\gamma}$ satisfies the assumptions above when considered on appropriate mean-zero 
function spaces, the proof
of \eqref{eq:Bassumption} seems to be problematic in this case. Notice that \eqref{eq:Bassumption} is used
only in the proof of \eqref{eq:visc_ode3}, so we bypass this difficulty by proving \eqref{eq:visc_ode3} directly,
relying on a technique developed in \cites{CZ15,CZKV15,CV12} for the surface quasi-geostrophic equation.

In the present case, the abstract space $H^1$ corresponds to the homogenous Sobolev space $\H^{\gamma/2}$.
Consider  the finite difference
\begin{align*}
\delta_hf^\nu(t,x)=f^\nu(t,x+h)-f^\nu(t,x),
\end{align*}
which is periodic in both $x$ and $h$, where  $x,h \in \T^d$.  In turn, 
\begin{align}\label{eq:findiff0}
L (\delta_hf^\nu)=0,
\end{align}
where $L$ denotes the differential operator
$$
L=\de_t+\u\cdot \nabla_x+(\delta_h\u)\cdot \nabla_h+ \nu\Lambda^\gamma.
$$
From \eqref{eq:findiff0}, we use the formula (see \cite{CC04})
\begin{align*}
2\varphi(x) \Lambda^\gamma \varphi(x)=\Lambda^\gamma \big(\varphi(x)^2\big)+D_\gamma[\varphi](x), 
\end{align*}
valid for $\gamma\in (0,2)$ and $\varphi\in C^\infty(\T^d)$, and with
\begin{align}
\label{eq:D:gamma:def}
D_\gamma[\varphi](x)= c_\gamma \int_{\R^d} \frac{\big[\varphi(x)-\varphi(x+y)\big]^2}{|y|^{d+\gamma}}\dd y.
\end{align}
We then arrive at
\begin{align}\label{eq:findiff}
L (\delta_hf^\nu)^2+ \nu D_\gamma[\delta_hf^\nu]=0.
\end{align}
For an arbitrary $\alpha\in (0,1)$,
we study the evolution of the quantity $v(t,x;h)$ defined by
\begin{align}
v(t,x;h) =\frac{\delta_hf^\nu(t,x)}{|h|^{d/2+\gamma/2}}.
\end{align}
Notice that $v$ is very much related to the usual homogeneous fractional  Sobolev norms, in the sense that
\begin{align}
\|f^\nu(t)\|^2_{\H^{\gamma/2}}=\int_{\R^d}\int_{\R^d}\big[v(t,x;h)\big]^2\dd h\, \dd x= 
\int_{\R^d}\int_{\R^d}\frac{\big[f^\nu(t,x+h)-f^\nu(t,x)\big]^2}{|h|^{d+\gamma}}\dd h\, \dd x.
\end{align}
From \eqref{eq:findiff}, pointwise in $x,h$ and $t$, we deduce that
\begin{align}
L v^2+\nu\frac{ D_\gamma[\delta_hf^\nu] }{|h|^{d+\gamma}}
&=-(d+\gamma) \frac{h}{|h|^2}\cdot \delta_h\u \, v^2\leq (d+\gamma)\|\nabla\u\|_{L_{t,x}^\infty}\, v^2. \label{eq:ineq1}
\end{align}
We integrate the above inequality first in $h\in \T^d$ and then $x\in \T^d$. Using that 
\begin{align}
\frac12\int_{\R^d}\int_{\R^d} \frac{ D_\gamma[\delta_hf^\nu] }{|h|^{d+\gamma}}\dd h\,\dd x
=\int_{\R^d}\int_{\R^d} \frac{|\delta_h\Lambda^{\gamma/2}f^\nu|^2}{|h|^{d+\gamma}}\dd h\, \dd x
=\|f^\nu\|^2_{\H^{\gamma}},
\end{align}
and the divergence free assumption we arrive at
\begin{align}\label{eq:Hgammavisc}
\ddt\|f^\nu\|^2_{\H^{\gamma/2}}+ 2\nu\|f^\nu\|^2_{\H^{\gamma}}\leq 2(d+\gamma)\|\nabla\u\|_{L_{t,x}^\infty}\|f^\nu\|^2_{\H^{\gamma/2}}.
\end{align}
This is precisely \eqref{eq:visc_ode3}.

\subsection{Linearized 2D Navier-Stokes equations around the Kolmogorov flow}
In this section, we consider the two-dimensional forced Navier-Stokes equations on a periodic domain 
\begin{align}
\T^2_L=\left[-\frac{\pi}{L},\frac{\pi}{L}\right]\times [-\pi,\pi],
\end{align}
for $L\geq1$, in the usual vorticity formulation
\begin{align}\label{eq:NSEvort}
\de_t\omega^\nu +\vv^\nu\cdot \nabla \omega^\nu=\nu\Delta \omega^\nu -\nu\cos y.
\end{align}
Here, $\nu >0$, and, denoting $\nabla^\perp=(-\de_y,\de_x)$, we have
\begin{align}
\vv^\nu= \nabla^\perp (\Delta)^{-1} \omega^\nu=(-\de_y \Delta^{-1} \omega^\nu,
\de_x \Delta^{-1} \omega^\nu)
\end{align}
A stationary solution to \eqref{eq:NSEvort} is given by
\begin{align}
\uu_S(x,y)=(u(y), 0)=(\sin y, 0), \qquad \omega^\nu_S(x,y)=-u'(y)=-\cos y.
\end{align}
Linearizing around this solution, namely writing
\begin{align}
\omega^\nu=f^\nu+\omega_S,
\end{align}
and neglecting the nonlinear contribution, we arrive at
\begin{align}\label{eq:Kolmolinear} 
\begin{cases}
\de_tf^\nu +\sin(y) \de_x \left[I+\Delta^{-1}\right] f^\nu=\nu\Delta f^\nu, \quad &\text{in } (x,y)\in  \T^2_L, \ t\geq 0,\\
f^\nu(0)=f^{in}, \quad &\text{in } (x,y)\in  \T^2_L,
\end{cases}
\end{align}
In agreement with the functional setting of Section \ref{sec:functional}, it is again convenient to expand $f^\nu$ 
as
\begin{align}
f^\nu(t,x,y)=\sum_{k\in \ZZ} f^\nu_k(t,y) \e^{iLkx}, \qquad f^\nu_{k}(t,y)=\frac{L}{2\pi} \int_{\T_L} f^\nu(t,x,y) \e^{-iLkx}\dd x,
\end{align}
and obtain
\begin{align}\label{eq:Kolmolinearfourier}
\begin{cases}
\de_tf_k^\nu +ikL\sin(y)  \left[I+\Delta_k^{-1}\right] f_k^\nu=\nu\Delta_k f_k^\nu,  \quad &\text{in } y\in  \T, \ t\geq 0,\\
f_k^\nu(0)=f_k^{in}, \quad &\text{in } y\in  \T,
\end{cases}
\end{align}
where
\begin{align}
\Delta_k= -L^2k^2+\de_{yy}.
\end{align}
We assume that
\begin{align}\label{eq:kLrel}
|k|\geq \begin{cases}
1 \quad \text{when}\quad L>1,\\
2 \quad \text{when}\quad L=1.
\end{cases}
\end{align}
As far as the functional spaces are concerned, we consider the $L^2$ based space
\begin{align}
H=\left\{\phi:\T\to \mathbb C:\ \| \phi\|^2_H=\int_\T \left[|\phi(y)|^2-|\Delta_k^{-1/2}\phi(y)|^2\right]\dd y<\infty \right\},
\end{align}
with scalar product
\begin{align}
\l \phi_1,\phi_2\r=\int_\T  (I+\Delta_k^{-1})\phi_1(y)\overline{\phi_2(y)}\dd y.
\end{align}
As a consequence of  \eqref{eq:kLrel}, the norm in $H$
is equivalent to the usual $L^2$-norm, and the operator
\begin{align}\label{eq:Bsin}
B=ikL\sin(y)  \left[I+\Delta_k^{-1}\right]
\end{align}
is skew-adjoint. As mentioned earlier, questions related to enhanced dissipation of solutions to \eqref{eq:Kolmolinear} have been addressed in many
recent papers \cites{LX17,WZZkolmo17,IMM17,BW13}, in which a time-scale proportional to $\nu^{-1/2}$ is achieved. 
With our approach, using the polynomial inviscid damping estimate in \cite{WZZkolmo17} of order $t^{-1}$, 
we obtain from Theorem \ref{thm:abspolymix} an enhanced dissipation time-scale 
proportional to $\nu^{-2/3}$. In fact, exploiting further structure (similar to that in spiral flow in Section \ref{sec:spirale} below), the time-scale
improves to $\nu^{-3/5}$.

\subsection{Linear kinetic theory}\label{rem:concrete3}
Within this framework, we can also consider convergence to equilibrium of linear kinetic equations. We study $g^\nu:(0,\infty)\times\T^d\times\R^d\to \R$ a probability distribution evolving in time by the linear kinetic equation at positive temperature and friction force proportional to $\nu>0$
\begin{align}
\de_t g^\nu +v\cdot \nabla_x g^\nu=\nu(\Delta_v+\nabla_v\cdot(v g^\nu)).
\end{align}
Using the detailed balance condition, we do the change of variables given by
\begin{align}
f^\nu=\frac{g^\nu}{G}-1,
\end{align}
where 
\begin{align}
G(v)=\frac{1}{(2\pi)^{d/2}}\e^{-\frac{|v|^2}{2}}
\end{align}
is the the equilibrium of the system. Under this change of variables the decay of $f^\nu$ is equivalent to the convergence of $g^\nu$ to the Gibbs measure $G$.

Further decomposing $f^\nu$ into its Fourier modes on the $x$ variable $\{f_k^\nu\}_{k\in\ZZ^d}$, we obtain the decoupled system of equations
\begin{align}
\de_t f_k^\nu +iv\cdot k f_k^\nu=\nu(\Delta_v f_k^\nu - v \cdot\nabla_v f_k^\nu).
\end{align}
In this setting, we define $H$ to be a weighted $L^2$ space. In particular, its inner product is given by
\begin{align}
\langle\phi,\psi\rangle_H=\int_{\R^d} \phi\,\overline{\psi}\,G(v)\dd v.
\end{align}
The operator $A\phi=-\Delta_v \phi+v\cdot\nabla_v \phi$ is self-adjoint, positive definite and compact on this inner product. Moreover, we have by definition
\begin{align}\label{eq:H1kinetic}
\|\phi\|_{H^1}^2=\langle A\phi,\phi\rangle=\int_{\R^d}|\nabla_v \phi|^2G(v)\dd v,
\end{align}
where the last identity follows from integrating by parts. Further, the operator $B \phi=iv\cdot k \phi$ is anti-symmetric with respect to this inner product. 
Checking the assumption \eqref{eq:Bassumption}, we have
\begin{align}
\langle B\phi,A\phi \rangle =i\int_{\R^d}v\cdot k|\nabla_v \phi|^2G(v)\dd v+ik\cdot \int_{\R^d} \phi\,\overline{\nabla_v \phi}\,G(v)\dd v.
\end{align}
We notice that the real part of the first term vanishes, hence applying H\"older's inequality in the second term and using \eqref{eq:H1kinetic} we get
\begin{align}
\left|\Re\langle B\phi,A\phi \rangle_H\right|\le |k|\|\phi\|_H\|\phi\|_{H^1}.
\end{align}
This inequality is in fact a stronger form of \eqref{eq:Bassumption}, and has analogies with the case of shear and spiral flows
analyzed  in Section~\ref{sec:passivescalars}.

\section{Enhanced dissipation in passive scalars}\label{sec:passivescalars}
In this section, we focus on solutions to advection diffusion equations as those presented earlier in Sections \ref{rem:concrete}
and \ref{rem:concrete2}. 

\subsection{Contact Anosov flows}\label{sub:Anosov}
We discuss here contact Anosov flows, a class of exponential mixing dynamical systems that satisfies the hypothesis of Theorem~\ref{thm:absexpmix}.
For this section, we take advantage of the results of Liverani found in \cite{Liv04}. 

Let $M$ be a $C^4$,  $2d+1$ dimensional connected compact Riemannian manifold with a metric $g$ and volume form $\omega$. Let $\Phi_t:M\to M$ be 
a $C^4$ flow, that is 
\begin{itemize}
	\item $\Phi_0$ is the identity on $M$,
	\item $\Phi_{t+\tau}=\Phi_t\circ \Phi_\tau$, for all $t,\tau \in \R$.
\end{itemize}
We assume that $\Phi_t$ is a contact Anosov flow, namely that the following holds:

\medskip

\noindent $\diamond$ \emph{Anosov property.}
At each point $x\in M$, there exists a splitting of the tangent
space $T_xM = E^s(x) \oplus E^c(x)\oplus E^u(x)$. The splitting is invariant with respect
to $\Phi_t$, $E^c$ is one dimensional and coincides with the flow direction; in addition
there exists $A, \mu > 0$ such that
\begin{align*}
\| \dd \Phi_t v\|_g\leq A \e^{-\mu t}\|v\|_g, \qquad \forall v\in E^s,\, t\geq 0,\\
\| \dd \Phi_t v\|_g\geq A \e^{\mu t}\|v\|_g, \qquad \forall v\in E^u,\, t\leq 0.
\end{align*}
Here,  $\|v\|_g=g(v,v)^{1/2}$ is the natural norm on the tangent space $TM$. 

\medskip

\noindent $\diamond$ \emph{Contact property.}
There exists a one form $\alpha$ such that $\alpha \wedge (\dd \alpha)^d$ is nowhere
zero, left-invariant by $\Phi_t$, that is $\alpha(\dd \Phi_tv)=\alpha(v)$ for each $t\in \R$ and each tangent vector $v\in TM$. Moreover, $\alpha \wedge (\dd \alpha)^d$ is the volume form $\omega$ of the manifold $M$.

\medskip

Concrete examples of contact Anosov flows are given by the geodesic flow in any negatively curved manifold. These flows have been classically studied see \cites{sinai1961geodesic,hopf1939statistik,anosov1967some}.

Given a contact Anosov flow $\Phi_t$, we define the underlying vector field $\uu:M\to TM$. Given a point $x\in M$, then $\uu(x)$ is defined as the velocity of the curve $\Phi_t(x):\R\to M$ at time equal to zero, that is to say
\begin{align}
\uu(x)=\left.\dot{\Phi}_t(x)\right|_{t=0}.
\end{align}
The flow $\Phi_t$ is the ODE flow induced by the vector field $\uu$,
\begin{align}
\dot{\Phi}_t(x)=\uu(\Phi_t(x)).
\end{align}
The contact property, in particular the hypothesis that the flow preserves the volume form, implies that $\uu$ is divergence free. Moreover, as the flow  $\Phi_t$ is $C^4$, the induced vector field $\uu$ is at least $C^3$.

We are interested in the mixing properties of the evolution induced by the flow. Namely, given a mean-free initial datum $f^{in}:M\to \R$, we consider $f(t,x):[0,\infty)\times M\to \R$ as the solution of
\begin{align}\label{eq:flowanosov}
\begin{cases}
\partial_t f(t,x)+g(\nabla_M f(t,x),\uu(x))=0,&\mbox{in $[0,\infty)\times M$}\\
f(0,x)=f^{in}(x), &\mbox{in $M$.}
\end{cases}
\end{align}
A consequence of the main result in \cite{Liv04} (see \cite{Liv04}*{Corollary 2.5}) is that  there exist positive constants $a_1$ and $a_2$ such that for each mean-free initial datum $f^{in}\in C^1$ and each $\varphi \in C^1$,
\begin{align}\label{eq:livexpmixing}
\left|\int_M f(t,x) \varphi(x) \dd\omega(x) \right| \leq a_1\| f^{in}\|_{C^1}\|\varphi\|_{C^1}\e^{-a_2 t}.
\end{align}
To be able to apply Theorem~\ref{thm:absexpmix}, we need to translate this decay of correlations into our $\H^{-1}$ framework. A consequence of this decay estimate is the following result, which is proven by approximation.

\begin{proposition}\label{prop:approxliv}
Given $f^{in}\in \H^1$, let $f$ be the solution to \eqref{eq:flowanosov} with initial datum $f^{in}$. There exists positive constants $\tilde{a}_1$ and $\tilde{a}_2$, such that
\begin{align}
\|f(t)\|_{\H^{-1}}\le \tilde{a}_1 \e^{-\tilde{a}_2 t}\|f^{in}\|_{\H^1}.
\end{align}
\end{proposition}
To prove the previous proposition we need the following approximation lemma.
\begin{lemma}\label{lem:approximation}
There exists a positive constant $C$ depending on $M$, such that given any $h_0\in \H^1$ and any $\eps>0$, there exists $h_\eps\in C^\infty(M)$, such that
\begin{align}\label{eq:gC1reg}
\|h_\eps\|_{C^1}\le C \eps^{-\frac{2d+1}{2}} \|h_0\|_{\H^1}
\end{align}
and
\begin{align}\label{eq:approximating}
\|h_\eps-h_0\|_{L^2}\le C \eps \|h_0\|_{\H^1}.
\end{align}
\end{lemma}
\begin{proof}[Proof of Lemma \ref{lem:approximation}]
We consider $h:[0,\infty)\times M\to \R$ the solution to the heat equation with initial datum $h_0$. That is to say, $h$ solves
\begin{align}
\begin{cases}
\partial_t h(t,x)=\Delta_M h(t,x)&\mbox{on $[0,\infty)\times M$}\\
h(0,x)=h_0(x)&\mbox{on $M$.}
\end{cases}
\end{align}
Given $\{\psi_k\}_{k\in \N}$ the orthonormal Fourier basis of $L^2$ (i.e. the eigenbasis of $-\Delta_M$), we can express
\begin{align}
h_\eps(x):= h(\eps^2,x)= \sum_{k\in \N} \e^{-\eps^2 \lambda_k} \langle h_0, \psi_k\rangle_{L^2}\psi_k(x).
\end{align}
First, we compute how well $h_\eps$ approximates $h_0$, that is
\begin{align}
\|h_\eps-h_0\|_{L^2}^2=\sum_{k\in \N} (1-\e^{-\eps^2 \lambda_k})^2 |\langle h_0, \psi_k\rangle_{L^2}|^2\le \eps^2 \sup_{\eta>0} \frac{(1-\e^{-\eta})^2}{\eta}\sum_{k\in \N}  \lambda_k|\langle h_0, \psi_k\rangle_{L^2}|^2= C \eps^2\|h_0\|_{\H^1}^2,
\end{align}
which shows \eqref{eq:approximating}.

Given $s\ge 1$, we compute the Sobolev norm
\begin{align}
\|h_\eps\|_{\H^s}^2=\sum_{k\in \N} \lambda_k^{s} \e^{-2\eps^2 \lambda_k} |\langle h_0, \psi_k\rangle_{L^2}|^2\le \sup_{\eta>0} \eta^{s-1}\e^{-2\eps^2\eta}  \sum_{k\in \N} \lambda_k|\langle h_0, \psi_k\rangle_{L^2}|^2= C\eps^{-2(s-1)}\|h_0\|_{\H^1}^2. 
\end{align}
Using that $M$ is compact and the Sobolev embedding $\H^s\hookrightarrow C^1$ with $s=(2d+1)/2+1$, we get that there exists $C$ such that
\begin{align}
\|h_\eps\|_{C^1}\le C \eps^{-\frac{2d+1}{2}}\|h_0\|_{\H^1}.
\end{align}
The proof is over.
\end{proof}
We now proceed with the proof of Proposition~\ref{prop:approxliv}.
\begin{proof}[Proof of Proposition~\ref{prop:approxliv}]
We fix $f^{in}$, $\varphi\in \H^1$, we consider $f^{in}_\eps$, $\varphi_\eps\in C^1$ given by the approximation Lemma~\ref{lem:approximation} and $f_\eps(t)$ the solution to \eqref{eq:flowanosov} with initial condition $f^{in}_\eps$. By the triangle inequality,
\begin{align}\label{aux1}
\left|\int_M f(t)\varphi \dd\omega(x)\right|\le \left|\int_M f_\eps(t)\varphi_\eps \dd\omega(x)\right|+\left|\int_M (f_\eps(t)-f(t))\varphi\dd\omega(x)\right|+\left|\int_M f_\eps(t)(\varphi-\varphi_\eps) \dd\omega(x)\right|.
\end{align}
Using that \eqref{eq:flowanosov} is linear, that it preserves the volume form, the approximation property \eqref{eq:approximating} and the Poincar\'e inequality, we obtain that
\begin{align}\label{aux2}
\left|\int_M (f_\eps(t)-f(t))\varphi\dd\omega(x)\right|\le \|f^{in}_\eps-f^{in}\|_H\|\varphi\|_H\le C \eps \|f^{in}\|_{\H^1}\|\varphi\|_{\H^1}.
\end{align}
Similarly,
\begin{align}\label{aux3}
\left|\int_M f_\eps(t)(\varphi-\varphi_\eps) \dd\omega(x)\right|\le C \eps \|f^{in}\|_{\H^1}\|\varphi\|_{\H^1}.
\end{align}
Using the $C^1$ mixing estimate \eqref{eq:livexpmixing} and the approximation property \eqref{eq:gC1reg}, we get
\begin{align}\label{aux4}
\left|\int_M f_\eps(t)\varphi_\eps \dd\omega(x)\right|\le a_1 \e^{-a_2 t}\|f^{in}_\eps\|_{C^1}\|\varphi\|_{C_1}\le Ca_1 \e^{-a_2 t}\eps^{-\frac{2d+1}{2}} \|f^{in}\|_{\H^1}\|\varphi\|_{\H^1}.
\end{align}
Using \eqref{aux1}, \eqref{aux2}, \eqref{aux3} and \eqref{aux4}, we get that for every $\eps>0$, we have the inequality
\begin{align}
\left|\int_M f(t)\varphi \dd\omega(x)\right|\le C(a_1 \e^{-a_2 t}\eps^{-\frac{2d+1}{2}}+\eps)\|f^{in}\|_{\H^1}\|\varphi\|_{\H^1}.
\end{align}
The result follows by optimizing the previous inequality over $\eps>0$ and using that the $\H^{-1}$ is the dual of $\H^1$. The proof is over.
\end{proof}

It is checked in Section~\ref{rem:concrete}, that the viscous evolution equation 
\begin{align}
\partial_t f^\nu+g(\uu,\nabla_M f^\nu)=\nu\Delta_M f^\nu
\end{align}
satisfies our basic assumptions. Hence, thanks to Theorem \ref{thm:absexpmix}, we infer that all contact Anosov flows are relaxation enhancing at logarithmic time-scale.

\begin{corollary}\label{cor:anos}
Let $\uu:M \to TM$ be the generator of a contact Anosov flow. Then 
\begin{align}
    \|f^{\nu}(t)\|_{L^2}\le \e^{-c_0|\ln \nu|^{-2} t}\|f^{in}\|_{L^2},\qquad \forall t>|\ln \nu|^{2}, 
\end{align}
In particular, $\uu$ is relaxation enhancing at logarithmic time-scale proportional to $|\ln\nu|^{2}$.
\end{corollary}

\subsection{Non-smooth Passive Scalars}

Another interesting case to consider is when the advecting velocity field $\uu$ is not bounded in $W^{1,\infty}_x$ but in $W^{1,r}_x$ for some $r<\infty$. Indeed, it remains an open problem whether there exists a velocity field $\uu\in L^\infty_t W^{1,\infty}_x(\mathbb{R}\times \mathbb{T}^2)$ which mixes any smooth initial datum $f^{in}$ exponentially. In a recent work of the third author and Zlatos \cite{EZ}, a velocity field $\uu$ is constructed with the property that for any $\tau_0$ the solution to 
\begin{align}
\label{passivescalar_EZ}
\begin{cases}
\partial_t f+\uu(t)\cdot\nabla f=0,&\text{on }  \T^2,\\
f(\tau_0)=f^{\tau_0},&\text{on } \T^2,
\end{cases}
\end{align}
with $f^{\tau_0}\in \H^1$ satisfy 
\begin{align}
\label{estimate_EZ}\|f(t)\|_{\H^{-1}}\leq \e^{-c(t-\tau_0)}\|f^{\tau_0}\|_{\H^1},
\end{align}
as $t\rightarrow\infty$ for some universal constant $c>0$. The velocity field constructed in \cite{EZ} belongs to $\uu\in L^\infty_tW^{1,r}$ for some $r>2$. A simple adaptation of the proof of Theorem \ref{thm:absexpmix} entails an enhanced dissipation estimate
for the viscous passive scalar problem
\begin{align}
\label{passivescalar_EZvisc}
\begin{cases}
\partial_t f^\nu+\uu(t)\cdot\nabla f^\nu=\nu\Delta f^\nu,&\text{on }  \T^2,\\
f(\tau_0)=f^{\tau_0},&\text{on } \T^2,
\end{cases}
\end{align}
as stated in the following theorem. 

\begin{theorem}
Let $\uu\in L^\infty_tW^{1,r}(\mathbb{R}\times\mathbb{T}^2)$ for some $r>2$. Assume that for all $\tau_0$ and $f^{\tau_0}\in \H^1$ the solution of \eqref{passivescalar_EZ} satisfies \eqref{estimate_EZ} for some fixed $c>0$. Then, there exist constants $C,c_0>0$ depending only on $c$ and $\|\uu\|_{L^\infty_t W^{1,r}_x}$ so that 
\begin{align}
  \label{dissipationW1p}  \| f^{\nu}(t)\|_{L^2}\le \e^{-c_0|\ln \nu|^{-2}\nu^{\frac{1}{r-1}} t}\|f^{in}\|_{L^2},\qquad \forall t>C\frac{|\ln \nu|^{2}}{\nu^\frac{1}{r-1}}.
\end{align} 
\end{theorem}

The proof of the previous Theorem follows the same arguments as Theorem \ref{thm:absexpmix}. The only difference is that the $\H^2$ estimate \eqref{eq:intest2} is replaced by 
\begin{align}
2\nu\int_{\tau_0}^{\tau_0+t}\|f^\nu(s)\|_{\H^2}^2\leq \|f^\nu(\tau_0)\|_{\H^1}^2+c_1\|\uu\|_{L^\infty_tW^{1,r}_x}\left(\int_{\tau_0}^{\tau_0+t}\|f^\nu(s)\|_{\H^1}^2\dd s\right)^{\frac{r-1}{r}}\left(\int_{\tau_0}^{\tau_0+t}\|f^\nu(s)\|_{\H^2}^2\dd s\right)^{\frac1r}
\end{align} 
for any $t,\tau_0\geq 0$.  
We now make a few remarks on this theorem. 

\begin{remark}
The constant $c_0$ is independent of $r$ and for the specific velocity field constructed in \cite{EZ}, we have 
that the velocity field is relaxation enhancing on the time scale $\nu^{-0.62}.$ We are also not aware of any velocity field on $\mathbb{T}^2,$ other than the one constructed in \cite{EZ}, which causes \emph{all} $L^2$ mean-free solutions of the advection diffusion equation to dissipate at a time scale of $\nu^{-q}$ for any $q<1$. 
\end{remark}
\begin{remark}
In \cite{ACM16}, a velocity field $\uu\in L^\infty_t W^{1,\infty}_x(\mathbb{R}\times\mathbb{T}^2)$, which mixes \emph{prescribed} initial data exponentially fast, is constructed. Since our arguments involve mixing of the inviscid evolution with arbitrary initial time and arbitrary initial data, it does not follow from our arguments that the velocity field constructed in \cite{ACM16} has an enhanced dissipation time scale. 
\end{remark}

\subsection{Shear flows with fractional dissipation}\label{sub:fracto}
We now return to the example addressed in Section \ref{rem:concrete2}, in the special case when $\uu$ is a shear flow.
To be precise, assume that $u \in C^{n_0+2}(\T)$ has a \emph{finite} number
of critical points, denoted by $\bar y_1,\ldots, \bar y_N$, and where  $n_0\in \N$ 
denotes the maximal order of vanishing of $u'$ at the critical points, namely, the minimal integer
such that
\begin{align}
u^{(n_0+1)}(\bar y_i)\neq 0, \qquad \forall i=1,\ldots N.
\end{align}
On the two-dimensional torus $\T^2$ and with $\uu=(u(y),0)$, the problem \eqref{eq:fractpass} now reads
\begin{align}\label{eq:fractshear}
\de_t f^\nu +u \de_x f^\nu+\nu\Lambda^\gamma f^\nu=0,
\end{align}
for $\gamma\in(0,2)$, while its inviscid counterpart is
\begin{align}\label{eq:fractshearinv}
\de_t f^\nu +u \de_x f^\nu=0.
\end{align}
As in the previous section, we can take the Fourier transform in the $x$-variable of the above equations, and study the problem
for each $x$-Fourier coefficient.
Via the method of stationary phase, in \cite{BCZ15}*{Theorem A.1} it was proven that the solution
 to the inviscid equation
\eqref{eq:fractshearinv} 
with initial datum $f^{in}\in L^2$ such that
\begin{align}\label{eq:mean0x}
\int_\T f^{in}(x,y) \dd x=0, \qquad \text{for a.e. } y\in \T,
\end{align}
satisfies  the mixing estimate
\begin{align}
\|f(t)\|_{\H^{-1}} & \leq \frac{a}{(1+t)^\frac{1}{n_0+1}}\|f^{in}\|_{H^1}, \qquad \forall t\geq 0,
\end{align}
for some constant $a>0$. Here, $\H^s$ refers to the usual homogeneous Sobolev space.
In addition, since $\|f(t)\|_{L^2}=\|f^{in}\|_{L^2}$, we can use standard interpolation theory (see \cite{Tartar}*{Lemma 22.3})
to deduce that 
\begin{align}
\|f_k(t)\|_{\H^{-\gamma/2}} & \leq \frac{a^{\gamma/2}}{(1+t)^\frac{\gamma}{2(n_0+1)}}\|f^{in}\|_{\H^{\gamma/2}}, \qquad \forall t\geq 0.
\end{align}
Hence, Theorem \ref{thm:abspolymix} directly applies to this case and allows us to deduce the following result.

\begin{proposition}\label{prop:enhancedshear}
Let $f^{in}\in L^2$ be such that
\begin{align}
\int_\T f^{in}(x,y) \dd x=0, \qquad \text{for a.e. } y\in \T.
\end{align}
There exists $c_\gamma>0$ such that for every $f^{in}\in H$ there holds the estimate 
\begin{align}
    \|f^{\nu}(t)\|_H\le \e^{-c_\gamma\nu^{q_{n_0,\gamma}} t}\|f^{in}\|_H,\qquad \forall t>\frac{1}{\nu^{q_{n_0,\gamma}}},
\end{align}
with 
\begin{align}\label{eq:qng}
q_{n_0,\gamma}=\frac{2}{2+\frac{\gamma}{2(n_0+1)}}.
\end{align}
In particular, the shear flow $u$, in the context of passive scalars with fractional dissipation, 
is relaxation enhancing with time-scale $O(1/\nu^{q_{n_0,\gamma}})$.
\end{proposition}
The condition \eqref{eq:mean0x} is the real-variable counterpart of considering $x$-Fourier coefficients corresponding to $k\neq 0$.
This is necessary, as the $x$-average of the solution satisfies the one-dimensional fractional heat equation, which clearly does not undergo
any enhanced dissipation.

\begin{remark}
Proposition \ref{prop:enhancedshear} is stated in real-variables for the sake of clarity, but one could also state a corresponding result
for each $x$-Fourier mode. Our choice is dictated by the form of the $\H^{\gamma/2}$ estimate \eqref{eq:Hgammavisc}, which does
not exploit any additional structure that may appear due to the simple form of the background flow 
(as opposed to the case of shear flows with ordinary diffusion $-\Delta$).  
\end{remark}

\begin{remark}
The case $\gamma=2$ was treated in \cite{BCZ15} using the so-called hypocoercivity method (see \cite{Villani09}). In particular,
the (optimal) enhanced dissipation time-scale found there was $1/\nu^\frac{n_0+1}{n_0+3}$, which does \emph{not} correspond to   
that deduced in \eqref{eq:qng} in the limit $\gamma\to 2$. However, the non-local nature of the fractional dissipation makes it difficult
to use hypocoercivity on this particular problem, since commutators between diffusion and advection are not as simple (for example, due
to the fact that the fractional laplacian $\Lambda$ only satisfies a generalized Leibniz rule).
This shows once again the generality and versatility of our method.
\end{remark}

\section{Spiral flows}\label{sec:spirale}
Let $B_1\subset \R^2$ be the open unit ball, centered at $(x,y)=(0,0)$, and consider the inviscid problem
\begin{align}\label{eq:inviscidspiral}
\begin{cases}
\de_t f+ \uu\cdot \nabla f=0, \quad &\text{in } B_1, \ t\geq 0,\\
f(0)=f^{in}, \quad &\text{in } B_1,
\end{cases}
\end{align}
where the mean-free scalar $f$ evolves under the autonomous velocity field
\begin{align}\label{eq:spiralflow}
\uu(r,\theta)=r^{1+\alpha}\begin{pmatrix}
-\sin\theta\\
\cos\theta
\end{pmatrix},
\end{align} 
and $\alpha\geq1$ is a real parameter.
It is worth noticing that $\uu$ is a stationary solution of the Euler equations, namely
\begin{align}
(\uu\cdot\nabla)\uu=\nabla P, \qquad P=\frac{|\uu|^2}{2(1+\alpha)}.
\end{align}
Stability of such radial solution in the 2D Euler equations were recently addressed in \cites{BCZV17,CZZ18, Zcirc17}.
From the mixing point of view in the passive scalar problem \eqref{eq:inviscidspiral}, the case $\alpha=1$ was recently studied in great 
detail in \cite{CLS17}, where the decay
of the $\H^{-1}$ and the geometric mixing scale was proven under the natural 
the natural condition orthogonality condition
\begin{align}\label{eq:ortho}
\int_{\de B_\rho} f^{in} \dd S_\rho=0,
\end{align}
for almost every $\rho>0$, where $\dd S_\rho$ is the uniform measure on the circle
of radius $\rho\in(0,1)$. By passing to polar coordinates $(r,\theta)\in [0,1)\times \T$ in  \eqref{eq:inviscidspiral}, we deduce that
\begin{align}\label{eq:inviscidspiralpol}
\begin{cases}
\de_t f+ r^{\alpha} \de_\theta f=0, \quad &\text{in } (r,\theta)\in [0,1)\times \T, \ t\geq 0,\\
f(0)=f^{in}, \quad &\text{in } (r,\theta)\in [0,1)\times \T.
\end{cases}
\end{align}
In this formulation, the analogies with the planar shear flow case become apparent.
By expanding the solution $f$ to \eqref{eq:inviscidspiralpol} as a Fourier series in the angular $\theta$ variable, namely
\begin{align}
f(t,r,\theta)=\sum_{k\in \ZZ} f_k(t,r)\e^{ik\theta}, \qquad f_k(t,r)=\frac{1}{2\pi}\int_0^{2\pi}f(t,r,\theta)\e^{-ik\theta}\dd \theta.
\end{align} 
for any integer $k$ we have that
\begin{align} \label{eq:inviscidspiralpolfourier1}
\begin{cases}
\de_t f_k+  ik r^{\alpha}f_k=0, \quad &\text{in } r\in [0,1), \ t\geq 0,\\
f_k(0)=f^{in}_{k}, \quad &\text{in } r\in [0,1).
\end{cases}
\end{align} 
Note that $f_0(t,r) = f^{in}(r)$ (i.e., the $\theta$-average of the solution is conserved), and therefore we restrict to $k\neq 0$ without loss
of generality. This is precisely the orthogonality condition \eqref{eq:ortho}. In the radial setting, we then define the usual Sobolev spaces in
terms of the Laplace operator
\begin{align}\label{eq:laplaceopk}
\Delta_k := \de_{rr} + \frac{1}{r}\de_r - \frac{k^2}{r^2}.
\end{align}
In particular, to be consistent with the notation of Section \ref{sec:functional}, we define
\begin{align}
H=\left\{\phi:[0,1)\to \mathbb C:\ \| \phi\|^2_H=\int_0^1 |\phi(r)|^2r\dd r<\infty \right\}, \qquad \l \phi_1,\phi_2\r=\int_0^1  \phi_1(r)\overline{\phi_2(r)}r\dd r,
\end{align}
and
\begin{align}\label{eq:H1normspiral}
H^1=\left\{\phi:[0,1)\to \mathbb C:\ \| \phi\|^2_{H^1}=\int_0^1 \left[|\de_r \phi(r)|^2+\frac{k^2}{r^2} |\phi(r)|^2\right]r\,\dd r<\infty\right\}.
\end{align}
The corresponding $H^{-1}$ norm is then defined by duality as
\begin{align}
\|\phi\|_{H^{-1}}=\sup_{\eta \in H^1: \|\eta\|_{H^1}=1}  \left|\int_0^1  \phi(r)\overline{\eta(r)}r \dd r\right|.
\end{align}
Notice that, since $k\neq 0$, we have
\begin{align}\label{eq:sob}
\|\phi\|_{L^\infty}\leq \| \phi\|_{H^1}.
\end{align}

\subsection{Mixing by spiral flows}\label{sub:spiralmix}

From the method of stationary phase (see  \cite{BigStein}*{Proposition 3, Chapter XIII}), we deduce the following mixing result.
\begin{proposition}[Mixing by spiral flows]\label{prop:spiralmix}
Fix $k \neq 0$, $\alpha\geq1$ and $f^{in}_{k}\in H^1$. Let $f_k$ be the solution to \eqref{eq:inviscidspiralpolfourier1}. Then
\begin{align}
\|f_k(t)\|_{H^{-1}} & \leq \frac{a_\alpha}{(1+|k|t)^{p_\alpha}}\|f^{in}_{k}\|_{H^1}, \qquad \forall t\geq 0,\label{ineq:spiraldecay}
\end{align}
where
\begin{align}
p_\alpha=\frac{2}{\max\{\alpha,2\}}
\end{align}
and $a_\alpha> 0$ is a constant independent of $k,t$ and $f^{in}$.
\end{proposition}

\begin{remark}
In physical space and for initial data satisfying \eqref{eq:ortho}, estimate \eqref{ineq:spiraldecay} can be written as
\begin{align}
\|f(t)\|_{H_{r,\theta}^{-1}} & \leq \frac{a_\alpha}{(1+t)^{p_\alpha}}\|f^{in}\|_{H_r^1H^{-1}_\theta}, \qquad \forall t\geq 0,
\end{align}
where the Sobolev norm in the $\theta$ variable is obtained in the usual Fourier sense. This estimate, in the special case $\alpha=1$,
sharpens the one obtained in \cite{CLS17}.
\end{remark}

 \begin{proof}[Proof of Proposition \ref{prop:spiralmix}]
 Since $\|f_k(t)\|_{H^{-1}}\leq \|f^{in}_{k}\|_{H}$ for every $t\geq 0$, it is enough to prove the above bound \eqref{ineq:spiraldecay} for 
 $t\geq1$.
From \eqref{eq:inviscidspiralpolfourier1}, we can write
\begin{align}
f_k(t,r)=\e^{-ik tr^{\alpha}} f^{in}_{k}(r).
\end{align}
Let $\eta \in H^1$ be such that $\|\eta\|_{H^1}=1$. We first deal with the case $\alpha\in [1,2)$. In this case
\begin{align}\label{eq:spiralflowfirstcase}
\int_0^1  \e^{-ik tr^{\alpha}} f^{in}_{k}(r)\overline{\eta(r)}r\dd r
&=-\frac{1}{\alpha ikt}\int_0^1 \frac{1}{ r^{\alpha-1}} \frac{\dd}{\dd r}\left(\e^{-ikt r^{\alpha}}\right) f^{in}_{k}(r)\overline{\eta(r)}r\dd r\notag\\
&= \frac{1}{\alpha ikt}\int_0^1  \e^{-ik tr^{\alpha}} \frac{\dd}{\dd r}\left[f^{in}_{k}(r)\overline{\eta(r)}r^{2-\alpha}\right]\dd r-\frac{r^{2-\alpha}}{\alpha ikt}  \e^{-ik tr^{\alpha}} f^{in}_{k}(r)\overline{\eta(r)}\bigg|_{r=0}^1.
\end{align}
Thus, in view of \eqref{eq:sob} and the restriction $\alpha\in [1,2)$, we have the bound on the second term
\begin{align}\label{eq:secondtermspiralflow}
\left|\frac{r^{2-\alpha}}{\alpha ikt}  \e^{-ik tr^{\alpha}} f^{in}_{k}(r)\overline{\eta(r)}\bigg|_{r=0}^1\right|\leq \frac{a_\alpha}{ |k|t}\|f^{in}_{k}\|_{H^1},
\end{align}
for some constant $a_\alpha>0$. Expanding the derivative in the first term of \eqref{eq:spiralflowfirstcase}, typical terms to bound are
\begin{align}
\left|\int_0^1 \de_rf^{in}_{k}(r)\overline{\eta(r)}r^{2-\alpha}\dd r\right|
\leq\left(\int_0^1|\de_rf^{in}_{k}(r)|^2r\dd r\right)^{1/2}\left(\int_0^1\frac{k^2}{r^2}|\eta(r)|^2 r\dd r\right)^{1/2}
\end{align}
and
\begin{align}
\left|\int_0^1 f^{in}_{k}(r)\overline{\eta(r)}r^{1-\alpha}\dd r \right|
\leq\frac{1}{|k|^2}\left(\int_0^1\frac{k^2}{r^2}|f^{in}_{k}(r)|^2 r\dd r\right)^{1/2}\left(\int_0^1\frac{k^2}{r^2}|\eta(r)|^2 r\dd r\right)^{1/2},
\end{align}
so that
\begin{align}\label{eq:firsttermspiralflow}
\left|\frac{1}{\alpha ikt}\int_0^1  \e^{-ik tr^{\alpha}} \frac{\dd}{\dd r}\left[f^{in}_{k}(r)\overline{\eta(r)}r^{2-\alpha}\right]\dd r\right|\leq
\frac{a_\alpha}{|k|t}\|f^{in}_{k}\|_{H^1},
\end{align}
where we have used the explicit form of the $H^1$ norm \eqref{eq:H1normspiral} and $|k|\ge 1$. Combining \eqref{eq:spiralflowfirstcase}, \eqref{eq:secondtermspiralflow}, \eqref{eq:firsttermspiralflow} and taking the supremum over all $\eta$ gives the result for $\alpha\in [1,2)$. 

We now deal with the case $\alpha\geq 2$. Let $\eps\in(0,1)$ to be fixed later. Then
\begin{align}
\left|\int_0^1  \e^{-ik r^{\alpha}} f^{in}_{k}(r)\overline{\eta(r)}r\dd r\right|
&\leq \left|\int_0^\eps  \e^{-ik r^{\alpha}} f^{in}_{k}(r)\overline{\eta(r)}r\dd r\right| + \left|\int_\eps^1  \e^{-ik r^{\alpha}} f^{in}_{k}(r)\overline{\eta(r)}r\dd r\right|.
\end{align}
Estimating the first piece, using \eqref{eq:sob} and that $\eta$ has unit norm gives
\begin{align}
\left|\int_0^\eps  \e^{-ik tr^{\alpha}} f^{in}_{k}(r)\overline{\eta(r)}r\dd r\right| \leq\eps^2  \|f^{in}_{k}\|_{H^1}.
\end{align}
Regarding the second piece, we integrate by parts as
\begin{align}
\int_\eps^1  \e^{-ik tr^{\alpha}} f^{in}_{k}(r)\overline{\eta(r)}r\dd r
&=-\frac{1}{\alpha ikt}\int_\eps^1 \frac{1}{ r^{\alpha-1}} \frac{\dd}{\dd r}\left(\e^{-ik tr^{\alpha}}\right) f^{in}_{k}(r)\overline{\eta(r)}r\dd r\notag\\
&= \frac{1}{\alpha ikt}\int_\eps^1  \e^{-ik tr^{\alpha}} \frac{\dd}{\dd r}\left[\frac{f^{in}_{k}(r)\overline{\eta(r)}}{r^{\alpha-2}}\right]\dd r-\frac{1}{\alpha ikt} \frac{1}{ r^{\alpha-2}} \e^{-ik tr^{\alpha}} f^{in}_{k}(r)\overline{\eta(r)}\bigg|_{r=\eps}^1.
\end{align}
Using \eqref{eq:sob}, it is not hard to see that
\begin{align}
\left|\int_\eps^1  \e^{-ik rt^{\alpha}} f^{in}_{k}(r)\overline{\eta(r)}r\dd r\right|\leq \frac{1}{\eps^{\alpha-2}} \frac{a_\alpha}{kt} \|f^{in}_{k}\|_{H^1},
\end{align}
for some constant $a_\alpha>0$, which is independent of $k$ and $t$. Therefore, for every $\eps\in (0,1)$, we have that
\begin{align}
\left|\int_0^1  f_k(t,r)\overline{\eta(r)}r\dd r\right|\leq \left[\eps^2+ \frac{1}{\eps^{\alpha-2}} \frac{a_\alpha}{kt} \right]\|f^{in}_{k}\|_{H^1},
\end{align}
so that, by optimizing in $\eps$, we find for $t\gg 1$ and up to an innocuous change in the constant $a_\alpha$ that
\begin{align}
\left|\int_0^1  f_k(t,r)\overline{\eta(r)}r\dd r\right|\leq \frac{a_\alpha}{(|k|t)^{2/\alpha}} \|f^{in}_{k}\|_{H^1}.
\end{align}
We now take the supremum in $\eta$ and obtain the desired result, concluding the proof.
 \end{proof}
 
\subsection{Enhanced dissipation in spiral flows}\label{sub:spiralenhanced}
 
When adding dissipation to \eqref{eq:inviscidspiralpol}, we obtain the advection diffusion equation
\begin{align}\label{eq:viscousspiralpol}
\begin{cases}
\de_t f^\nu+ r^{\alpha} \de_\theta f^\nu=\nu\Delta_{r,\theta} f^\nu, \quad &\text{in } (r,\theta)\in [0,1)\times \T, \ t\geq 0,\\
f^\nu(0)=f^{in}, \quad &\text{in } (r,\theta)\in [0,1)\times \T,
\end{cases}
\end{align}
where we supplement the system with the classical no-flux boundary condition
\begin{align}
\de_r f^\nu(1,\theta)=0, \qquad \forall \theta \in \T.
\end{align}
Analogously to  \eqref{eq:inviscidspiralpolfourier1}, we obtain
\begin{align} \label{eq:viscousspiralpolfourier1}
\begin{cases}
\de_t f^\nu_k+  ik r^{\alpha}f^\nu_k=\nu\Delta_kf^\nu_k, \quad &\text{in } r\in [0,1), \ t\geq 0,\\
f^\nu_k(0)=f^{in}_{k}, \quad &\text{in } r\in [0,1),
\end{cases}
\end{align} 
where $\Delta_k$ is defined in \eqref{eq:laplaceopk}, and with boundary conditions
\begin{align}\label{eq:nofluxBC}
\de_r f_k^\nu(1)=0, \qquad \forall k \in \ZZ.
\end{align}
In view of Proposition \ref{prop:spiralmix}, it is clear that we could apply 
directly Theorem \ref{thm:abspolymix} and deduce an enhanced dissipation time-scale. Instead, we prefer to 
exploit the particular structure of \eqref{eq:viscousspiralpolfourier1} and obtain and even faster time-scale.

\begin{theorem}\label{thm:enhancespiral}
For each $k\neq 0$ and each $\nu\in(0,1)$ such that $\nu|k|^{-1}<1$, consider the passive scalar problem \eqref{eq:viscousspiralpolfourier1}. Then, for every $f^{in}\in H$ there holds the estimate 
\begin{align}
    \|f^{\nu}_k(t)\|_H\le \e^{-c_0\nu^{q_\alpha} |k|^{1-q_\alpha} t}\|f^{in}_k\|_H,\qquad \forall t>\frac{1}{\nu^{q_\alpha}|k|^{1-q_\alpha}},
\end{align}
with 
\begin{align}
q_\alpha=\frac{4-p_\alpha}{4+p_\alpha}, \qquad  p_\alpha=\frac{2}{\max\{\alpha,2\}},
\end{align}
and
\begin{align}
 c_\alpha=\frac{1}{128}\min\left\{\frac{1}{64(1+\alpha^2)},\frac{1}{a_\alpha 4^{p_\alpha}}\right\}.
\end{align}
In particular, the spiral flow $\uu$ in \eqref{eq:spiralflow}, in the context of passive scalars, 
is  relaxation enhancing with time-scale $O(1/\nu^{q_\alpha})$.
\end{theorem}

\begin{remark}
The above result is an enhanced dissipation estimate for each spherical harmonic of the solution to \eqref{eq:viscousspiralpol},
in which the dependence of the time-scale is made precise also in the angular frequency $k\neq 0$. The constraint 
$\nu |k|^{-1}<1$ is very natural, and relevant in hypoelliptic problems, because it can help to quantify the regularization of the
solution in the angular variable, even if dissipation in the angular direction is not present (see \cite{BCZ15} for the planar shear flow case).
\end{remark}

\begin{proof}[Proof of Theorem \ref{thm:enhancespiral}]
We begin by treating the case $k=1$. We will see at the end that the general case will follow with
a simple time rescaling.
In order to further lighten the notation in the proof, we omit the dependence on $k=1$ of the 
solutions to \eqref{eq:inviscidspiralpolfourier1} and \eqref{eq:viscousspiralpolfourier1}.
Clearly, \eqref{eq:visc_ode} holds in the same way, as
\begin{align}\label{eq:basicener}
\ddt \|f^{\nu}\|^2_{H}+2\nu \| f^\nu\|^2_{H^1}=0.
\end{align}
Besides the dependence on $k$ of all the constants, the main difference with the proof of Theorem \ref{thm:abspolymix}
consists in an estimate on the operator $B=i  r^\alpha $ that improves the second inequality in \eqref{eq:Bassumption}.
Indeed, using the antisymmetry of $B$ and the boundary condition \eqref{eq:nofluxBC}, we have
\begin{align}
|\Re\l Bf^\nu, \Delta f^\nu\r|
=|\Re\l ir^\alpha f^\nu, (\de_{rr} +r^{-1} \de_r)f^\nu\r |
=\alpha|\Re\l ir^{\alpha-1} f^\nu, \de_{r} f^\nu\r |,
\end{align}
implying
\begin{align}
|\Re\l Bf^\nu, \Delta f^\nu\r|\leq \alpha  \| f^\nu\|_H\|f^\nu\|_{H^1}.
\end{align}
In turn (compare with \eqref{eq:visc_ode2}-\eqref{eq:visc_ode3}), we deduce that
\begin{align}\label{eq:H1spafsfsa}
\ddt \|f^{\nu}\|^2_{H^1}+2\nu \| \Delta f^\nu\|^2_{H}\leq 2 \alpha  \| f^{\nu}\|_H\|f^\nu\|_{H^1}.
\end{align}
while \eqref{eq:close1} then becomes
\begin{align}\label{eq:close1spiral}
 \| f^\nu(t)-f(t)\|^2_{H}
&\leq\| f^\nu(\tau_0)-f^{\tau_0}\|^2_{H}\notag\\
&\quad+ \sqrt{ t}\left(4\alpha \nu \|f^\nu(\tau_0)\|_{H}\int_{\tau_0}^{\tau_0+t}\|f^\nu(s)\|_{H^1}\dd s+2\nu\|f^{\nu}(\tau_0)\|^2_{H^1}\right)^{1/2}\| f^{\tau_0}\|_{H}\notag\\
&\leq\| f^\nu(\tau_0)-f^{\tau_0}\|^2_{H}\notag\\
&\quad+ \sqrt{ t}\left(4\alpha \nu \|f^\nu(\tau_0)\|_{H}\sqrt{ t}\left(\int_{\tau_0}^{\tau_0+t}\|f^\nu(s)\|^2_{H^1}\dd s\right)^{1/2}+2\nu\|f^{\nu}(\tau_0)\|^2_{H^1}\right)^{1/2}\| f^{\tau_0}\|_{H}.
\end{align}
Mimicking the contradiction argument of Theorem \ref{thm:abspolymix},
we first show that for all $\nu<1$, we have the inequality 
\begin{align}\label{decay delta1spiral}
\nu\int_{0}^{\nu^{-q_\alpha}}\|f^\nu(t)\|_{H^1}^2\dd t\geq \delta_\alpha\|f^{in}\|^2_{H},
\end{align}
where
\begin{align}\label{choice of delta1spiral}
    \delta_\alpha=\frac{1}{64}\min\left\{\frac{1}{64(1+\alpha^2)},\frac{1}{a_\alpha 4^{p_\alpha}}\right\}.
\end{align}
As above, assuming $\|f^{in}\|_{H}=1$, there exists a $\tau_1\in [0,\nu^{-q_\alpha}]$ so that 
\begin{align}\label{eq:deftau11spiral}
\nu\int_{\tau_1}^{\tau_1+\nu^{-\frac{1+q_\alpha}{4}}}\|f^\nu(s)\|_{H^1}^2\dd s< 2\delta_\alpha \nu^{q_\alpha-\frac{1+q_\alpha}{4}},
\end{align}
from which we infer the existence of $\tau_0\in [\tau_1,\tau_1+\nu^{-\frac{1+q_\alpha}{4}}/2]$ such that 
\begin{align}\label{reversepoicareattau01spiral}
\nu\|f^\nu(\tau_0)\|_{H^1}^2< 4\delta_\alpha\nu^{q_\alpha} 
\end{align}
and
\begin{align}\label{boundtau01spiral}
    \nu\int_{\tau_0}^{\tau_0+\nu^{-\frac{1+q_\alpha}{4}}/2}\|f^\nu(s)\|_{H^1}^2\dd s< 2\delta_\alpha \nu^{q_\alpha-\frac{1+q_\alpha}{4}}.
\end{align}
Now we take $f^{\tau_0}=f^\nu(\tau_0)$ as initial datum for the inviscid problem \eqref{eq:inviscidspiralpolfourier1} for $k=1$ and with initial time $\tau_0$ and denote the solution by $f(t+\tau_0)$ with $t\geq 0$. Using that $\|f^\nu(\tau_0)\|_H\leq 1$, the estimate on the proximity of the two flows \eqref{eq:close1spiral}, the properties of $\tau_0$ \eqref{reversepoicareattau01spiral} and \eqref{boundtau01spiral}, we have
\begin{align}\label{eq:close2_1spiral}
\| f^\nu(\tau_0+t)-f(\tau_0+t)\|^2_{H} 
&\leq\sqrt{ t}\left(4\alpha \nu \sqrt{ t}\left(\int_{\tau_0}^{\tau_0+t}\|f^\nu(s)\|^2_{H^1}\dd s\right)^{1/2}+2\nu\|f^{\tau_0}\|^2_{H^1}\right)^{1/2}\notag\\
&\leq \sqrt{\frac{ \nu^{-\frac{1+q_\alpha}{4}}}{2}}\left(8\alpha\delta_\alpha\nu^{\frac{1+q_\alpha}{4}} +8\delta_\alpha\nu^{q_\alpha}\right)^{1/2}\leq\frac{1}{4},
\end{align}
for all $t\in [0,\frac12 \nu^{-\frac{1+q_\alpha}{4}}]$, thanks to our choice of $\delta_\alpha$ and using that $\nu<1$. Now, following
the proof of Theorem \ref{thm:abspolymix}, we use the mixing estimate \eqref{ineq:spiraldecay} and \eqref{reversepoicareattau01spiral} to find that for any $R\geq 1$ and any $t\in [\frac{1}{4}\nu^{-\frac{1+q_\alpha}{4}},\frac{1}{2}\nu^{-\frac{1+q_\alpha}{4}}]$
the inviscid problem satisfies 
\begin{align}\label{eq:mixlowspiral}
\| P_{\leq R}f(\tau_0+t)\|^2_{H} 
\leq\frac{a_\alpha^2 R}{t^{2p_\alpha}}\|f^{\tau_0}\|^2_{H^1} \leq  4^{2p_\alpha+1}a_\alpha^2 \delta_\alpha\nu^{\frac{q_\alpha+1}{2}p_\alpha+q_\alpha-1}R.
\end{align}
Consequently, 
\begin{align}\label{eq:highfreq3_11spiral}
\| f^\nu(\tau_0+t)\|^2_{H^1}&\geq \frac{R}{8}\left(1 -4^{2p_\alpha+2}a_\alpha^2 \delta_\alpha\nu^{\frac{q_\alpha+1}{2}p_\alpha+q_\alpha-1}R\right), \qquad \forall t\in \left[\frac{1}{4}\nu^{-\frac{1+q_\alpha}{4}},\frac{1}{2}\nu^{-\frac{1+q_\alpha}{4}} \right],
\end{align}
and an optimization in $R$ leads to
\begin{align}
\| f^\nu(\tau_0+t)\|^2_{H^1}\geq\frac{1}{512}\frac{1}{4^{2p}a_\alpha^2 \delta_\alpha\nu^{\frac{q_\alpha+1}{2}p_\alpha+q_\alpha-1}}, \qquad \forall t\in \left[\frac{1}{4}\nu^{-\frac{1+q_\alpha}{4}},\frac{1}{2}\nu^{-\frac{1+q_\alpha}{4}} \right].\label{eq:finalestforH1spiral}
\end{align}
Integrating over $\left(\frac{1}{4}\nu^{-\frac{1+q_\alpha}{4}},\frac{1}{2}\nu^{-\frac{1+q_\alpha}{4}}\right)$ and using the bound \eqref{boundtau01spiral}, we obtain
\begin{align}
2\delta_\alpha \nu^{q_\alpha-\frac{1+q_\alpha}{4}} >\frac{1}{2048}\frac{\nu^{-\frac{1+q_\alpha}{4}}}{4^{2p}a^2 \delta_\alpha\nu^{\frac{q_\alpha+1}{2}p_\alpha+q_\alpha-1}}.
\end{align}
By re-arranging and recalling that $q_\alpha(p_\alpha+4)-4+p_\alpha=0$, we get  that 
\begin{align}
\delta_\alpha^2>\frac{1}{4096}\frac{1}{4^{2p} a_\alpha^2},
\end{align}
which contradicts our choice of $\delta_\alpha$ \eqref{choice of delta1spiral} and proves the desired estimate \eqref{decay delta1spiral}.
The conclusion (for $k=1$) then follows in the same way as in Theorem \ref{thm:abspolymix}. Now, if $k\neq 1$,
it is not hard to see that if $f^\nu_k$ solves \eqref{eq:viscousspiralpolfourier1}, then 
\begin{align}
\tilde{f}^\nu(t,r)=f^\nu_k(|k|t,r)
\end{align}
solves
\begin{align}
\begin{cases}
\de_t \tilde{f}^\nu+\mathrm{sign}(k)  i r^{\alpha}\tilde{f}^\nu=\nu |k|^{-1}\Delta \tilde{f}^\nu , \quad &\text{in } r\in [0,1), \ t\geq 0,\\
\tilde{f}^\nu(0)=f^{in}_{k}, \quad &\text{in } r\in [0,1),
\end{cases}
\end{align} 
Hence, the above proof applies to $\tilde{f}^\nu$ (the extra factor $\mathrm{sign}(k)$ is irrelevant for the mixing estimate \eqref{ineq:spiraldecay}) and we obtain that 
\begin{align}
\|f^\nu_k(t)\|=\|\tilde{f}^\nu(|k|t)\|_H\leq \e^{-c_0\nu^{q_\alpha} |k|^{1-q_\alpha} t} \|f^{in}_{k}\|_H,
\end{align}
concluding the proof.
\end{proof}

\subsection*{Acknowledgements}
The authors would like to thank the following people for helpful discussions: Jacob Bedrossian, Oliver Butterley, Gautam Iyer, Lucia Simonelli.
MCZ was partially supported by NSF grant DMS-1713886. MGD was partially supported by EPSRC grant number EP/P031587/1.


\begin{bibdiv}
\begin{biblist}

\bib{ACM16}{article}{
      author={{Alberti}, G.},
      author={{Crippa}, G.},
      author={{Mazzucato}, A.~L.},
       title={{Exponential self-similar mixing by incompressible flows}},
        date={2016-05},
     journal={ArXiv e-prints},
      eprint={1605.02090},
}

\bib{anosov1967some}{article}{
      author={Anosov, Dmitry~Victorovich},
      author={Sinai, Yakov~Grigor'evich},
       title={Some smooth ergodic systems},
        date={1967},
     journal={Russian Mathematical Surveys},
      volume={22},
      number={5},
       pages={103\ndash 167},
}

\bib{BajerEtAl01}{article}{
      author={Bajer, Konrad},
      author={Bassom, Andrew~P},
      author={Gilbert, Andrew~D},
       title={Accelerated diffusion in the centre of a vortex},
        date={2001},
     journal={Journal of Fluid Mechanics},
      volume={437},
       pages={395\ndash 411},
}

\bib{BDL18}{article}{
      author={Baladi, Viviane},
      author={Demers, Mark~F.},
      author={Liverani, Carlangelo},
       title={Exponential decay of correlations for finite horizon sinai
  billiard flows},
        date={2018},
     journal={Invent. Math.},
      volume={211},
      number={1},
       pages={39\ndash 177},
}

\bib{BW13}{article}{
      author={Beck, Margaret},
      author={Wayne, C.~Eugene},
       title={Metastability and rapid convergence to quasi-stationary bar
  states for the two-dimensional {N}avier-{S}tokes equations},
        date={2013},
     journal={Proc. Roy. Soc. Edinburgh Sect. A},
      volume={143},
      number={5},
       pages={905\ndash 927},
         url={http://dx.doi.org/10.1017/S0308210511001478},
}

\bib{BCZV17}{article}{
      author={{Bedrossian}, J.},
      author={{Coti Zelati}, M.},
      author={{Vicol}, V.},
       title={{Vortex axisymmetrization, inviscid damping, and vorticity
  depletion in the linearized 2D Euler equations}},
        date={2017-11},
     journal={ArXiv e-prints},
      eprint={1711.03668},
}

\bib{BGM15I}{article}{
      author={{Bedrossian}, J.},
      author={{Germain}, P.},
      author={{Masmoudi}, N.},
       title={{Dynamics near the subcritical transition of the 3D Couette flow
  I: Below threshold case}},
        date={2015-06},
     journal={ArXiv e-prints},
      eprint={1506.03720},
}

\bib{BGM15II}{article}{
      author={{Bedrossian}, J.},
      author={{Germain}, P.},
      author={{Masmoudi}, N.},
       title={{Dynamics near the subcritical transition of the 3D Couette flow
  II: Above threshold case}},
        date={2015-06},
     journal={ArXiv e-prints},
      eprint={1506.03721},
}

\bib{BVW16}{article}{
      author={Bedrossian, J.},
      author={Vicol, V.},
      author={Wang, F.},
       title={The {S}obolev stability threshold for {2D} shear flows near
  {C}ouette},
        date={2016},
     journal={To appear in J. Nonlin. Sci.. Preprint: arXiv:1604.01831},
}

\bib{BCZ15}{article}{
      author={Bedrossian, Jacob},
      author={Coti~Zelati, Michele},
       title={Enhanced dissipation, hypoellipticity, and anomalous small noise
  inviscid limits in shear flows},
        date={2017},
     journal={Arch. Ration. Mech. Anal.},
      volume={224},
      number={3},
       pages={1161\ndash 1204},
}

\bib{BCZGH15}{article}{
      author={Bedrossian, Jacob},
      author={Coti~Zelati, Michele},
      author={Glatt-Holtz, Nathan},
       title={Invariant {M}easures for {P}assive {S}calars in the {S}mall
  {N}oise {I}nviscid {L}imit},
        date={2016},
     journal={Comm. Math. Phys.},
      volume={348},
      number={1},
       pages={101\ndash 127},
}

\bib{BGM15III}{article}{
      author={Bedrossian, Jacob},
      author={Germain, Pierre},
      author={Masmoudi, Nader},
       title={On the stability threshold for the 3{D} {C}ouette flow in
  {S}obolev regularity},
        date={2017},
     journal={Ann. of Math. (2)},
      volume={185},
      number={2},
       pages={541\ndash 608},
}

\bib{BM15}{article}{
      author={Bedrossian, Jacob},
      author={Masmoudi, Nader},
       title={Inviscid damping and the asymptotic stability of planar shear
  flows in the 2{D} {E}uler equations},
        date={2015},
     journal={Publ. Math. Inst. Hautes \'Etudes Sci.},
      volume={122},
       pages={195\ndash 300},
         url={http://dx.doi.org/10.1007/s10240-015-0070-4},
}

\bib{BMV14}{article}{
      author={Bedrossian, Jacob},
      author={Masmoudi, Nader},
      author={Vicol, Vlad},
       title={Enhanced dissipation and inviscid damping in the inviscid limit
  of the {N}avier-{S}tokes equations near the two dimensional {C}ouette flow},
        date={2016},
     journal={Arch. Ration. Mech. Anal.},
      volume={219},
      number={3},
       pages={1087\ndash 1159},
         url={http://dx.doi.org/10.1007/s00205-015-0917-3},
}

\bib{Bressan03}{article}{
      author={Bressan, Alberto},
       title={A lemma and a conjecture on the cost of rearrangements},
        date={2003},
     journal={Rend. Sem. Mat. Univ. Padova},
      volume={110},
       pages={97\ndash 102},
}

\bib{CKRZ08}{article}{
      author={Constantin, P.},
      author={Kiselev, A.},
      author={Ryzhik, L.},
      author={Zlatos, A.},
       title={Diffusion and mixing in fluid flow},
        date={2008},
     journal={Ann. of Math. (2)},
      volume={168},
      number={2},
       pages={643\ndash 674},
         url={http://dx.doi.org/10.4007/annals.2008.168.643},
}

\bib{CV12}{article}{
      author={Constantin, Peter},
      author={Vicol, Vlad},
       title={Nonlinear maximum principles for dissipative linear nonlocal
  operators and applications},
        date={2012},
     journal={Geom. Funct. Anal.},
      volume={22},
      number={5},
       pages={1289\ndash 1321},
         url={http://dx.doi.org/10.1007/s00039-012-0172-9},
}

\bib{CC04}{article}{
      author={C{\'o}rdoba, Antonio},
      author={C{\'o}rdoba, Diego},
       title={A maximum principle applied to quasi-geostrophic equations},
        date={2004},
     journal={Comm. Math. Phys.},
      volume={249},
      number={3},
       pages={511\ndash 528},
         url={http://dx.doi.org/10.1007/s00220-004-1055-1},
}

\bib{CZZ18}{article}{
      author={{Coti Zelati}, M.},
      author={{Zillinger}, C.},
       title={{On degenerate circular and shear flows: the point vortex and
  power law circular flows}},
        date={2018-01},
     journal={ArXiv e-prints},
      eprint={1801.07371},
}

\bib{CZ15}{article}{
      author={Coti~Zelati, Michele},
       title={Long-{T}ime {B}ehavior and {C}ritical {L}imit of {S}ubcritical
  {SQG} {E}quations in {S}cale-{I}nvariant {S}obolev {S}paces},
        date={2018},
     journal={J. Nonlinear Sci.},
      volume={28},
      number={1},
       pages={305\ndash 335},
}

\bib{CZKV15}{article}{
      author={Coti~Zelati, Michele},
      author={Kalita, Piotr},
       title={Smooth attractors for weak solutions of the {SQG} equation with
  critical dissipation},
        date={2017},
     journal={Discrete Contin. Dyn. Syst. Ser. B},
      volume={22},
      number={5},
       pages={1857\ndash 1873},
}

\bib{CLS17}{article}{
      author={{Crippa}, G.},
      author={{Luc{\`a}}, R.},
      author={{Schulze}, C.},
       title={{Polynomial mixing under a certain stationary Euler flow}},
        date={2017-07},
     journal={ArXiv e-prints},
      eprint={1707.09909},
}

\bib{Deng2013}{article}{
      author={Deng, Wen},
       title={Resolvent estimates for a two-dimensional non-self-adjoint
  operator},
        date={2013},
     journal={Commun. Pure Appl. Anal.},
      volume={12},
      number={1},
       pages={547\ndash 596},
}

\bib{DubrulleNazarenko94}{article}{
      author={Dubrulle, B},
      author={Nazarenko, S},
       title={On scaling laws for the transition to turbulence in uniform-shear
  flows},
        date={1994},
     journal={Euro. Phys. Lett.},
      volume={27},
      number={2},
       pages={129},
}

\bib{Gallay2017}{article}{
      author={{Gallay}, T.},
       title={{Enhanced dissipation and axisymmetrization of two-dimensional
  viscous vortices}},
        date={2017-07},
     journal={ArXiv e-prints},
      eprint={1707.05525},
}

\bib{EZ}{article}{
      author={{Elgindi}, T.~M.},
      author={{Zlatos}, A.},
       title={{On universal mixers in all dimensions}},
        date={2018},
     journal={Preprint},
      eprint={},
}

\bib{GNRS18}{article}{
      author={{Grenier}, E.},
      author={{Nguyen}, T.~T.},
      author={{Rousset}, F.},
      author={{Soffer}, A.},
       title={{Linear inviscid damping and enhanced viscous dissipation of
  shear flows by using the conjugate operator method}},
        date={2018-04},
     journal={ArXiv e-prints},
      eprint={1804.08291},
}

\bib{hopf1939statistik}{book}{
      author={Hopf, Eberhard},
       title={{Statistik der geod{\"a}tischen Linien in Mannigfaltigkeiten
  negativer Kr{\"u}mmung}},
        date={1939},
}

\bib{IMM17}{article}{
      author={{Ibrahim}, S.},
      author={{Maekawa}, Y.},
      author={{Masmoudi}, N.},
       title={{On pseudospectral bound for non-selfadjoint operators and its
  application to stability of Kolmogorov flows}},
        date={2017-10},
     journal={ArXiv e-prints},
      eprint={1710.05132},
}

\bib{IKX14}{article}{
      author={Iyer, Gautam},
      author={Kiselev, Alexander},
      author={Xu, Xiaoqian},
       title={Lower bounds on the mix norm of passive scalars advected by
  incompressible enstrophy-constrained flows},
        date={2014},
     journal={Nonlinearity},
      volume={27},
      number={5},
       pages={973\ndash 985},
}

\bib{Jabin16}{article}{
      author={Jabin, Pierre-Emmanuel},
       title={Critical non-{S}obolev regularity for continuity equations with
  rough velocity fields},
        date={2016},
     journal={J. Differential Equations},
      volume={260},
      number={5},
       pages={4739\ndash 4757},
}

\bib{Kelvin87}{article}{
      author={Kelvin, Lord},
       title={Stability of fluid motion: rectilinear motion of viscous fluid
  between two parallel plates},
        date={1887},
     journal={Phil. Mag.},
      volume={24},
      number={5},
       pages={188\ndash 196},
}

\bib{LatiniBernoff01}{article}{
      author={Latini, M.},
      author={Bernoff, A.J.},
       title={Transient anomalous diffusion in {Poiseuille} flow},
        date={2001},
     journal={Journal of Fluid Mechanics},
      volume={441},
       pages={399\ndash 411},
}

\bib{LiWeiZhang2017}{article}{
      author={{Li}, T.},
      author={{Wei}, D.},
      author={{Zhang}, Z.},
       title={{Pseudospectral and spectral bounds for the Oseen vortices
  operator}},
        date={2017-01},
     journal={ArXiv e-prints},
      eprint={1701.06269},
}

\bib{LWZ18}{article}{
      author={{Li}, T.},
      author={{Wei}, D.},
      author={{Zhang}, Z.},
       title={{Pseudospectral bound and transition threshold for the 3D
  Kolmogorov flow}},
        date={2018-01},
     journal={ArXiv e-prints},
      eprint={1801.05645},
}

\bib{LX17}{article}{
      author={{Lin}, Z.},
      author={{Xu}, M.},
       title={{Metastability of Kolmogorov flows and inviscid damping of shear
  flows}},
        date={2017-07},
     journal={ArXiv e-prints},
      eprint={1707.00278},
}

\bib{LTD11}{article}{
      author={Lin, Zhi},
      author={Thiffeault, Jean-Luc},
      author={Doering, Charles~R.},
       title={Optimal stirring strategies for passive scalar mixing},
        date={2011},
     journal={J. Fluid Mech.},
      volume={675},
       pages={465\ndash 476},
         url={http://dx.doi.org/10.1017/S0022112011000292},
}

\bib{Liv04}{article}{
      author={Liverani, Carlangelo},
       title={On contact {A}nosov flows},
        date={2004},
     journal={Ann. of Math. (2)},
      volume={159},
      number={3},
       pages={1275\ndash 1312},
}

\bib{LLNMD12}{article}{
      author={Lunasin, Evelyn},
      author={Lin, Zhi},
      author={Novikov, Alexei},
      author={Mazzucato, Anna},
      author={Doering, Charles~R.},
       title={Optimal mixing and optimal stirring for fixed energy, fixed
  power, or fixed palenstrophy flows},
        date={2012},
     journal={J. Math. Phys.},
      volume={53},
      number={11},
       pages={115611, 15},
}

\bib{MV11}{article}{
      author={Mouhot, Cl{\'e}ment},
      author={Villani, C{\'e}dric},
       title={On {L}andau damping},
        date={2011},
     journal={Acta Math.},
      volume={207},
      number={1},
       pages={29\ndash 201},
         url={http://dx.doi.org/10.1007/s11511-011-0068-9},
}

\bib{RS79-3}{book}{
      author={Reed, Michael},
      author={Simon, Barry},
       title={Methods of modern mathematical physics. {III}},
   publisher={Academic Press [Harcourt Brace Jovanovich, Publishers], New
  York-London},
        date={1979},
        note={Scattering theory},
}

\bib{RhinesYoung83}{article}{
      author={Rhines, P.B.},
      author={Young, W.R.},
       title={How rapidly is a passive scalar mixed within closed
  streamlines?},
        date={1983},
     journal={Journal of Fluid Mechanics},
      volume={133},
       pages={133\ndash 145},
}

\bib{Seis13}{article}{
      author={Seis, Christian},
       title={Maximal mixing by incompressible fluid flows},
        date={2013},
     journal={Nonlinearity},
      volume={26},
      number={12},
       pages={3279\ndash 3289},
}

\bib{sinai1961geodesic}{inproceedings}{
      author={Sinai, Ya~G},
       title={Geodesic flows on compact surfaces of negative curvature},
        date={1961},
   booktitle={Dokl. akad. nauk sssr},
      volume={136},
       pages={549\ndash 552},
}

\bib{BigStein}{book}{
      author={Stein, Elias~M.},
       title={Harmonic analysis: real-variable methods, orthogonality, and
  oscillatory integrals},
      series={Princeton Mathematical Series},
   publisher={Princeton University Press, Princeton, NJ},
        date={1993},
      volume={43},
        note={With the assistance of Timothy S. Murphy, Monographs in Harmonic
  Analysis, III},
}

\bib{Tartar}{book}{
      author={Tartar, Luc},
       title={An introduction to {S}obolev spaces and interpolation spaces},
      series={Lecture Notes of the Unione Matematica Italiana},
   publisher={Springer, Berlin; UMI, Bologna},
        date={2007},
      volume={3},
}

\bib{Villani09}{article}{
      author={Villani, C{\'e}dric},
       title={Hypocoercivity},
        date={2009},
     journal={Mem. Amer. Math. Soc.},
      volume={202},
      number={950},
       pages={iv+141},
         url={http://dx.doi.org/10.1090/S0065-9266-09-00567-5},
}

\bib{WZ18}{article}{
      author={{Wei}, D.},
      author={{Zhang}, Z.},
       title={{Transition threshold for the 3D Couette flow in Sobolev space}},
        date={2018-03},
     journal={ArXiv e-prints},
      eprint={1803.01359},
}

\bib{WZZkolmo17}{article}{
      author={{Wei}, D.},
      author={{Zhang}, Z.},
      author={{Zhao}, W.},
       title={{Linear inviscid damping and enhanced dissipation for the
  Kolmogorov flow}},
        date={2017-11},
     journal={ArXiv e-prints},
      eprint={1711.01822},
}

\bib{WZZ17}{article}{
      author={{Wei}, D.},
      author={{Zhang}, Z.},
      author={{Zhao}, W.},
       title={{Linear inviscid damping and vorticity depletion for shear
  flows}},
        date={2017-04},
     journal={ArXiv e-prints},
      eprint={1704.00428},
}

\bib{WZZ15}{article}{
      author={Wei, Dongyi},
      author={Zhang, Zhifei},
      author={Zhao, Weiren},
       title={Linear inviscid damping for a class of monotone shear flow in
  {S}obolev spaces},
        date={2018},
     journal={Comm. Pure Appl. Math.},
      volume={71},
      number={4},
       pages={617\ndash 687},
}

\bib{YZ17}{article}{
      author={Yao, Yao},
      author={Zlatos, Andrej},
       title={Mixing and un-mixing by incompressible flows},
        date={2017},
     journal={J. Eur. Math. Soc. (JEMS)},
      volume={19},
      number={7},
       pages={1911\ndash 1948},
}

\bib{Zillinger18}{article}{
      author={{Zillinger}, C.},
       title={{On geometric and analytic mixing scales: comparability and
  convergence rates for transport problems}},
        date={2018-04},
     journal={ArXiv e-prints},
      eprint={1804.11299},
}

\bib{Zillinger15}{article}{
      author={Zillinger, Christian},
       title={Linear inviscid damping for monotone shear flows in a finite
  periodic channel, boundary effects, blow-up and critical {S}obolev
  regularity},
        date={2016},
     journal={Arch. Ration. Mech. Anal.},
      volume={221},
      number={3},
       pages={1449\ndash 1509},
}

\bib{Zillinger14}{article}{
      author={Zillinger, Christian},
       title={Linear inviscid damping for monotone shear flows},
        date={2017},
     journal={Trans. Amer. Math. Soc.},
      volume={369},
      number={12},
       pages={8799\ndash 8855},
}

\bib{Zcirc17}{article}{
      author={Zillinger, Christian},
       title={On circular flows: linear stability and damping},
        date={2017},
     journal={J. Differential Equations},
      volume={263},
      number={11},
       pages={7856\ndash 7899},
}

\end{biblist}
\end{bibdiv}
\end{document}